\documentclass[12pt,reqno]{amsart}%%%%%%%%%%%%{article}
\usepackage{etex}
\usepackage {amssymb}
\usepackage {amsmath}
\usepackage {bbm}
\usepackage{amsthm}
\usepackage{mathtools}
\usepackage{graphicx}
\usepackage {amscd}
\usepackage {epic}
\usepackage[alphabetic]{amsrefs}
\usepackage{hyperref}
\usepackage{enumerate}
\usepackage{tikz}
\usepackage{tikz-cd}
\usepackage{verbatim,color,geometry}
\usepackage[all]{xy}
\usepackage{enumitem}
\usepackage{bm}
\newcommand{\red}[0]{\operatorname{red}}

\newcommand{\Z}{{\mathbb{Z} }}

\newcommand{\cX}{{\mathcal{X}}}

\newcommand{\cY}{{\mathcal{Y}}}
\newcommand{\cZ}{{\mathcal{Z}}}

\newcommand{\Supp}{{\rm Supp}}

\newcommand{\ord}{{\rm ord}}
\newcommand{\gr}{{\rm gr}}

\newtheorem{thm}{Theorem}[section]

\newtheorem{lem}[thm]{Lemma}
\newtheorem{cor}[thm]{Corollary}

\newtheorem{prop}[thm]{Proposition}

\theoremstyle{definition}
\newtheorem{defn}[thm]{Definition}

\newtheorem{rem}[thm]{Remark}

\newtheorem{defn-thm}[thm]{Definition--Theorem}  %!!!!!!!!!!!!!!!!!!!!!!!!
\newtheorem{defn-prop}[thm]{Definition--Proposition}  %!!!!!!!!!!!!!!!!!!!!!!!!
\newtheorem{defn-lem}[thm]{Definition--Lemma}  %!!!!!!!!!!!!!!!!!!!!!!!!
\theoremstyle{remark}

\newcommand{\Proj}{{\rm Proj}}
\newcommand{\Val}{{\rm Val}}
\newcommand{\vol}{{\rm vol}}

\newcommand{\cD}{{\mathcal{D}}}

\newcommand{\cO}{{\mathcal{O}}}

\newcommand{\fa}{{\mathfrak{a}}}
\newcommand{\lct}{{\rm lct}}

\newcommand{\bQ}{{\mathbb{Q}}}
\newcommand{\Q}{{\mathbb{Q}}}
\newcommand{\N}{{\mathbb{N}}}
\newcommand{\fb}{{\mathfrak{b}}}
\newcommand{\fX}{\mathfrak{X}}

\newcommand{\Spec}{{\rm Spec}}

\renewcommand{\d}{\delta}

\newcommand{\Exc}{{ \rm Exc}}

\newcommand{\la}{\lambda}

\newcommand{\cF}{\mathcal{F}}

\newcommand{\QM}{{\rm QM}}

\newcommand{\R}{{\mathbb{R}}}

\newcommand{\CX}[1]{{\textcolor{blue}{[Chenyang: #1]}}}

\newcommand{\fm}{\mathfrak{m}}
\newcommand{\cL}{\mathcal{L}}
\newcommand{\tc}{\mathrm{tc}}
\newcommand{\bin}{\mathbf{in}}
\newcommand{\bG}{\mathbb{G}}
\newcommand{\bA}{\mathbb{A}}

\newcommand{\sslash}{\mathbin{/\mkern-6mu/}}
\newcommand{\bmu}{\bm{\mu}}

\numberwithin{equation}{section}
\newcommand\numberthis{\addtocounter{equation}{1}\tag{\theequation}}

\begin{document}

\pagenumbering{arabic}

\title[Openness of K-semistability]{Openness of K-semistability for Fano varieties
}
\date{\today}

\author{Harold Blum}
\address{Department of Mathematics, Stony Brook University, Stony Brook, NY 11794, USA}
\email{harold.blum@stonybrook.edu}

\author{Yuchen Liu}
\address{Department of Mathematics, Northwestern University, Evanston, IL 60208, USA.}
%\email{yuchen.liu@yale.edu}
%\email{yuchenl@math.princeton.edu}
\email{yuchenl@northwestern.edu}

\author {Chenyang Xu}
\address{Department of Mathematics, Princeton University, Princeton, NJ 08544, USA.}
\email{chenyang@princeton.edu}
\address   {MIT,       Cambridge 02139, USA}
\email     {cyxu@math.mit.edu}
\address   {Beijing International Center for Mathematical Research,       Beijing 100871, China}
\email     {cyxu@math.pku.edu.cn}

\begin{abstract}{
In this paper, we prove the openness of K-semistability in families of log Fano pairs by showing that the stability threshold is a constructible function on the fibers.
We also prove that any special test configuration arises from a log canonical place of a bounded complement and establish properties of any minimizer of the stability threshold. 
}
\end{abstract}

\maketitle{}
\setcounter{tocdepth}{1}
\tableofcontents

{\let\thefootnote\relax\footnotetext{
HB was partially supported by NSF grant DMS-1803102. 
YL was partially supported by the Della Pietra Endowed Postdoctoral Fellowship of the MSRI (NSF No. DMS-1440140).
CX was partially supported by a Chern Professorship of the MSRI (NSF No. DMS-1440140), NSF grant DMS-1901849 and FRG Grant DMS-1952531.
}
\marginpar{}
}

\section{Introduction}\label{s-intro}

 Throughout the paper, we work over an algebraically closed characteristic zero field. 
\smallskip

K-stability was invented as an algebraic condition to characterize when a Fano variety admits a K\"ahler-Einstein metric (see \cite{Tia97, Don02}). In  recent years, the question of whether one can construct the moduli space parametrizing K-polystable $\bQ$-Fano varieties with fixed numerical invariants, as well as establish nice properties for it, has attracted significant interest.  Previously, the construction of the moduli space  relied on the properties, especially the existence, of K\"ahler-Einstein metrics (see e.g. \cite{LWX19}). Nevertheless, using the valuative criterion developed in \cite{Fuj19, Li17}, a purely algebro-geometric approach has been dramatically advanced. See \cite{BX18, ABHLX19} for more background. 

This paper aims to settle one of the main steps of the construction. Namely, we prove that in a $\mathbb{Q}$-Gorenstein family $(X,\Delta)\to B$ of log Fano pairs over a normal base, the locus $B^{\circ}\subset B$  parametrizing K-semistable fibers is a Zariski open set. This is the last ingredient needed to conclude that the moduli space of K-polystable $\mathbb{Q}$-Fano varieties exists. See Theorem \ref{t-moduli} for a more precise statement.

\subsection{Main theorems}
Before proving the openness result, we first establish the following property of the stability threshold as well as Tian's $\alpha$-invariant. 

\begin{thm}\label{t-constructible}
If $(X,\Delta)\to B$ is a $\Q$-Gorenstein family of log Fano pairs
over a normal base $B$, then
the functions
\[
B \ni b \mapsto \min\{ \alpha(X_{\overline{b}},\Delta_{\overline{b}}),1\}
\quad
\text{ and }
\quad
B \ni b \mapsto \min\{\delta(X_{\overline{b}},\Delta_{\overline{b}}),1 \}
\]
are constructible and lower semicontinuous. 
\end{thm}
Recall that the $\alpha$-invariant of a log Fano pair $(X,\Delta)$ was introduced in \cite{Tia87} and the stability threshold (also known as the $\delta$-invariant) in \cite{FO18}. 
It was shown in \cite{FO18,BJ17} that $\delta(X,\Delta)\ge 1$ if and only if $(X,\Delta)$ is K-semistable, based on the valuative criteria for K-semistability proved in \cite{Fuj19, Li17}. Therefore, we have the following immediate corollary of Theorem \ref{t-constructible}.

\begin{cor}\label{c-openness}
If $(X,\Delta)\to B$ is a  $\Q$-Gorenstein 
family of log Fano pairs
over a normal base $B$, then 
\[
B^\circ: = \{ b \in B \, \vert \, (X_{\overline{b}},\Delta_{\overline{b}})
\text{ is  K-semistable} \}
\]
is a Zariski open subset of $B$. 
\end{cor}

Together with the main theorems in \cite{Jia19, BX18, ABHLX19}, we deduce 
\begin{thm}\label{t-moduli}
 The moduli functor $\mathfrak{X}^{\rm Kss}_{V,n}$ of K-semistable $\Q$-Fano varieties of dimension $n$ and volume $V$ is an Artin stack of finite type over $k$ and admits a separated good moduli space $\mathfrak{X}^{\rm Kss}_{V,n} \to X^{\rm Kps}_{V,n}$, whose  $k$-points parameterize  K-polystable $\Q$-Fano varieties of dimension $n$ and volume $V$.
\end{thm}

In fact, 
the boundedness of $\mathfrak{X}^{\rm Kss}_{V,n}$
was settled in \cite{Jia19}, which heavily relied on results in \cite{Bir16a}. Corollary \ref{c-openness} then implies that  the moduli functor $\mathfrak{X}^{\rm Kss}_{V,n}$ is an Artin stack of finite type over $k$. 
With the latter step completed, it follows from the main theorems in \cite{BX18, ABHLX19} that $\mathfrak{X}^{\rm Kss}_{V,n}$ admits the separated good moduli space $\mathfrak{X}^{\rm Kss}_{V,n} \to X^{\rm Kps}_{V,n}$.

As a consequence of Corollary \ref{c-openness} and Theorem \ref{t-moduli}, we show that K-stability (resp. K-polystability) is an open (resp. constructible) condition for $\bQ$-Gorenstein families of log Fano pairs; see Theorem \ref{t-openkstable}.

\begin{rem}
An analogue of Theorem \ref{t-constructible}
in a local setting was proved in \cite{Xu19} 
for the normalized volume function defined in \cite{Li18}. 
Corollary \ref{c-openness} and Theorem \ref{t-moduli} can also be obtained independently as a consequence of the local result via the cone construction. 
%Our proof of Theorem \ref{t-constructible} and Corollary \ref{c-openness} does not depend on \cite{Xu-19}.
\end{rem}

% Theorem \ref{t-constructible} follows from  Theorem \ref{t-deltaplace}.1 for the case when $\delta(X,\Delta)<1$, which  
\subsection{Outline of the proof}
Our strategy of proving Theorem \ref{t-constructible} is approximating the infimum 
$$\delta(X_{\overline{b}},\Delta_{\overline{b}})=\inf_{E} \frac{A_{X_{\overline{b}},\Delta_{\overline{b}}}(E)}{S(E)}~~\mbox{ \ for all divisors $E$ over $X_{\overline{b}}$}$$ by the values on lc places $E$ of bounded complements. We then deduce  constructibility
 by using a theorem on invariance of log  plurigenera established in \cite[Theorem 1.8]{HMX13}.

\medskip

More precisely,  the proof of Theorem \ref{t-constructible} relies on combining two techniques. The first one is the special degeneration theory initiated in \cite{LX14} and later developed in \cite{LX16, Fuj19, Fuj17, BX18} etc. Roughly speaking,  to compute $\delta(X,\Delta)$ for a log Fano pair $(X,\Delta)$,  instead of testing general divisorial valuations, we can focus on a special class of valuations, which are those that arise from  a special degeneration. It was known previously that it suffices to  consider such degenerations for  studying K-stability. Our new strategy, which is the second ingredient in this paper, is to use global complements to study them. 

The concept of a {\it complement} was introduced in \cite{Sho92}. Since then it has been a particularly effective  tool in birational geometry for understanding Fano varieties. In particular, a profound theorem on the existence of bounded global complements for log Fano pairs was proved by \cite{Bir16a}.  By using the techniques from \cite{Fuj17, LX16}, one can show that the valuation computing $\min\{\delta, 1\}$ can be approximated by special divisors (see \cite{BLZ19, ZZ19}). By applying Birkar's result, we deduce that all these special divisors are lc places of a bounded family of complements.
%\footnote{\YL{Indeed, this paper does not directly show special approximation to $\delta$ (proved in [BLZ19]) but rather approximation by lc places of $N$-complements as a combination of [BLZ19, ZZ19] and [Bir19]. Shall we change the phrasing? Do we want to mention [BLZ19, ZZ19] in the intro or not?}} 

The above discussion can be easily extended to a $\mathbb{Q}$-Gorenstein family of log Fano pairs $(X,\Delta)\to B$, and finally we can use \cite{HMX13} to conclude that 
$$b\mapsto \frac{A_{X_{b},\Delta_{b} }(E_{b})}{S(E_{b})}$$ 
is a constant function if the special divisor $E_{b}$ over $(X_b,\Delta_b) $ varies in a family giving fiberwise log resolutions. 

The arguments in Section 1-4 are of a global nature. In the appendix, we will develop this strategy  further using local techniques.

\subsection{Appendix}

In Appendix \ref{s-weaklyspecial}, we will use complements to further study the K-stability of a log Fano pair $(X,\Delta)$. The results proved in Appendix \ref{s-weaklyspecial} are not needed elsewhere in this paper. However, we expect it will be useful for future research.

%We will first give a characterization of all weakly special semiample test configurations (see Theorem \ref{thm:weaklyspecial}), as well as verify that minimizers of $\frac{A_{X,\Delta}(v)}{S(v)}$ computes the log canonical threshold of a bounded complement whenever $\delta(X,\Delta)\le 1$ . 

We first prove the following theorem which gives a characterization of a valuation $v$ computing $\delta(X,\Delta)$ when $\delta(X,\Delta)\le 1$.

\begin{thm}[{=Theorem \ref{t-minimizer=lcplaceN}}]\label{t-deltaplace}
Let $n$ be a positive integer and $I\subset \bQ$ a finite set. Then there exists a positive integer $N=N(n,I)$ satisfying the following:

Let $(X,\Delta)$ be an $n$-dimensional log Fano pair such that coefficients of $\Delta$ belong to $I$.
 If $\delta(X,\Delta)\le 1$, and $v$ is a valuation computing $\delta(X,\Delta)$, then $v$ is quasi-monomial and an lc place of an $N$-complement.
\end{thm}

While part of Theorem \ref{t-deltaplace}  can be proved by a global method similar to our proof of Theorem \ref{t-constructible} (see Proposition \ref{p-minapprox} and Theorem \ref{t-quasimonomialless1}), the statement in the full generality has to be established in a somewhat different way.  For this we have to invoke the cone construction and use some arguments from \cite{Xu19}. The technique of using the cone construction to study the K-stability of a log Fano pair was initiated in \cite{Li17} and played a key role in proving results in \cite{LX16, LWX18, BX18}.

We also show the following theorem which gives a characterization of weakly special test configurations. It is obtained by combining arguments in \cite{LWX18} and \cite{Xu19}, which uses the existence of bounded local complements. 

\begin{thm}[{=Theorem \ref{thm:weaklyspecial}}]\label{thm-weaklyspecial}
Let $n$ be a positive integer and $I\subset \bQ$ a finite set. Then there exists a positive integer $N=N(n,I)$ satisfying the following:

If $(X,\Delta)$ is an $n$-dimensional log Fano pair such that coefficients of $\Delta$ belong to $I$, then a finite set of $\Z$-valued divisorial valuations
$\{v_1,\cdots, v_d\}\subset\Val_X$ is a weakly special collection (see Definition \ref{d-weaklyspecial})
if and only if there exists an $N$-complement $\Delta^+$ of $(X,\Delta)$ such that each $v_i$ is an lc place of 
$(X,\Delta^+)$.
\end{thm}
By \cite{LX14}, to study K-(semi,poly)stability, we can concentrate on the class of weakly special test configurations.   Theorem \ref{thm-weaklyspecial} says that this class of  test  configurations comes from a somewhat  `bounded' amount of information. 
 %it is already known that there is correspondence between 

\bigskip{\bf Postscript remarks.} Since the first version of this article appeared on the arXiv, there has been works generalizing and strengthening our results. We list a few related works below.

\begin{enumerate}
    \item In \cite[Theorem 1.1]{LXZ21}, it is shown that any valuation computing $\delta(X,\Delta)<\frac{n+1}{n}$ for an $n$-dimensional log Fano pair $(X,\Delta)$ has a finitely generated associated graded ring. Theorem \ref{t-deltaplace} is a crucial step in proving this result. This result together with \cite{BHLLX20, XZ19} implies that the the K-moduli space $X_{V,n}^{\rm Kps}$ is a projective scheme.
    \item In \cite[Corollary 3.7]{LXZ21}, it is shown in the setting of Theorem \ref{t-constructible} that $B\ni b\mapsto \min\{\delta(X_{\overline{b}},\Delta_{\overline{b}}), \frac{n+1}{n}\}$ is constructible and lower semicontinuous, where $n$ is the relative dimension of $X/B$.
    \item Zhuang found a characterization of special prime divisors over a log Fano pair, that are, prime divisors induced by special test configurations in \cite[Theorem 4.12]{Xu20} as a strengthening of Theorem \ref{thm-weaklyspecial}.
\end{enumerate}

\bigskip{\bf Acknowledgement:}  We thank Davesh Maulik, Chuyu Zhou, and Ziquan Zhuang for helpful discussions. 
We also would like to thank the anonymous referees for many useful comments.  
Much of the work on this paper was completed while the
authors enjoyed the hospitality of the MSRI, which is gratefully acknowledged.

\section{Preliminaries}\label{s-prelim}

\subsection{Conventions} 
%Throughout, we work over an algebraically closed  characteristic 0 field $k$.
We will follow standard  terminologies in \cite{KM98, Kol13}.
A \emph{(normal) pair} $(X,\Delta)$ is composed of a normal variety $X$ and an effective $\Q$-divisor $\Delta$
on $X$ such that $K_X+\Delta$ is $\mathbb{Q}$-Cartier. 
See \cite[2.34]{KM98} for the definitions of \emph{klt}, \emph{plt}, and \emph{lc} pairs.

A pair $(X,\Delta)$ is \emph{log Fano} if $X$ is projective, $(X,\Delta)$ is klt, and $-K_X-\Delta$ is ample. 
A variety $X$ is {\it $\mathbb{Q}$-Fano} if $(X,0)$ is log Fano. More generally, a variety $X$ is of {\it Fano type} if it is projective and  there exists a $\Q$-divisor  $\Delta$ such that $(X,\Delta)$ is klt and $-K_X-\Delta$ is big and nef. 

For a $\mathbb{Q}$-divisor $L$, we write $|L|_{\mathbb{Q}}$ for the set effective $\bQ$-divisors which are $\bQ$-linearly equivalent to $L$.
For a subset  $I \subseteq [0,1]$, we set
$$I_{+}=\{0\}\bigcup \Big\{ j\in [0,1] \ \Big\vert \ j=\sum^l_{p=1}{i}_p \mbox{ for some }i_1,...,i_l \in I \Big\}$$ and 
$$D(I)=\Big\{\frac{m-1+a}{m}\ \Big\vert \ a\in I_+ \mbox{ and } m\in\N\big\}.$$

\subsection{Families of pairs}

\begin{defn}
A \emph{$\Q$-Gorenstein family of (normal) pairs} $f:(X,\Delta) \to B$  over a normal base is the data of a flat surjective morphism of varieties $f:X \to B$ and a
$\Q$-divisor $\Delta$ on $X$ satisfying
\begin{enumerate}
\item  $B$ is normal and $f$ has normal, connected fibers (hence, $X$ is normal as well),
\item $\Supp(\Delta)$ does not contain a fiber, and
\item $K_{X/B} +\Delta$ is $\Q$-Cartier. 
\end{enumerate}
We say $(X,\Delta)\to B$ is a \emph{$\Q$-Gorenstein family of log Fano pairs} if in addition
 $(X_b,\Delta_b)$ is log Fano for all $b\in B$. Here, $\Delta_b$ is the cycle pull-back of $\Delta$ to the fiber $X_b$. See \cite[Section 4]{Kol19} for more background.
 \end{defn}

In birational geometry, we should  usually allow the fibers to be slc pairs. However, in this note, we are only interested in families whose fibers are of Fano type. Thus, we can assume all fibers are normal.

\begin{defn}
Let  $f:(X,\Delta)\to B$ be a $\Q$-Gorenstein  family of pairs with $B$ smooth. 
A morphism $g:Y\to X$ is a \emph{fiberwise log resolution} of $(X,\Delta) \to B$ if $Y$ is smooth over $B$,  $E : = \sum_{i\in I} E_i  =\Exc(g) + \Supp( g_*^{-1}D)$ is an snc divisor, and each stratum of $E$ is smooth with irreducible fibers over $B$.
(Here, the \emph{strata} of $E$ are the irreducible components of $E_{J} = \cap_{j\in J} E_j$, 
for some subset $J\subseteq I$.)
\end{defn}

If $(X,\Delta) \to B$ is a $\Q$-Gorenstein family of pairs, then we can always find a nonempty open set $U\subset B$ and a finite \'etale map $U'\to U$ such  that $(X_{U'}, \Delta_{U'})\to U'$ admits a fiberwise log resolution.

\subsection{Valuations}
Let $X$ be a variety. A valuation on $X$ will mean a valuation $v: K(X)^\times \to \mathbb{R}$  that is trivial on $k$ and has center on $X$. 
Recall, $v$ has center on $X$ if there exists a point $\xi\in X$ such that $v\geq 0$ on $\cO_{X,\xi}$ and $>0$ on $\fm_{\xi} \subset \cO_{X,\xi}$. Since $X$ is assumed to be separated, such a point $\xi$ is unique, and we say $v$ has \emph{center} $c_X(v): = \xi$. 
 If $X$ is proper, then the valuative criterion for properness implies such a center always exists uniquely.
By convention, we set $v(0) =+\infty$.

Following  \cites{JM12,BdFFU15}, we write $\Val_{X}$ for the set of valuations on $X$ and $\Val_X^*$ for the set of non-trivial ones. We endow $\Val_{X}$ with the topology of pointwise convergence.

To any valuation $v\in \Val_{X}$ and $p\in \N$, there is an associated \emph{valuation ideal} $\fa_{p}(v)$.
For an affine open subset $U\subseteq X$, $\fa_p(v)(U) = \{ f\in \cO_X(U) \, \vert \, v(f) \geq p \}$ if $c_X(v) \in U$ and   $\fa_p(v)(U) =  \cO_X(U)$ otherwise.
 
For an ideal $\fa \subseteq \cO_X$ and $v\in \Val_X$, we set
\[
v(\fa) : = \min \{ v(f) \, \vert\, f\in \fa \cdot \cO_{X,c_X(v)} \} \in [0, +\infty].\]
 We can also make sense of $v(s)$ when $\cL$ is a line bundle and $s\in H^0(X,\cL)$. After trivializing $\cL$ at $c_X(v)$, we set $v(s)$ equal to  the value of the local function corresponding to $s$ under this trivialization; this is independent of the choice of trivialization. 
 
Similarly, if $D$ is a Cartier divisor, we set $v(D):=v(f)$, where $f$ is a local equation for $D$ at $c_X(v)$. If $D$ is only $\Q$-Cartier, we set $v(D) : = m^{-1} v(mD)$, where $m$ is a positive integer so that $mD$ is Cartier.
 
\subsubsection{Divisorial valuations}
Let $\mu:Y\to X$ be a proper birational morphism of varieties
with $Y$ normal. A prime divisor  $E\subset Y$ (called a \emph{prime divisor over} $X$) induces a valuation $\ord_E : K(X)^\times \to \Z$ given by order of vanishing along $E$. A valuation of the form $c\cdot \ord_E$, where $c\in \Q_{>0}$, is called \emph{divisorial}. We we write ${\rm DivVal}_X\subset \Val_X$ for the set of such valuations. 

\subsubsection{Quasi-monomial valuations}\label{ss-qm}
Let $\mu:Y \to X$ be a proper birational morphism with $Y$ regular.
Fix a not necessarily closed point $\eta \in Y$ and $y_1,\ldots,y_r$ a regular system of parameters for $\cO_{Y,\eta}$. 
Given $\alpha =(\alpha_1,\ldots,\alpha_r) \in \mathbb{R}_{\geq0}^r$, we define a valuation $v_\alpha$ as follows: 
For $f \in \cO_{Y,\eta}$, 
we can write $f$ in $\widehat{\cO_{Y,\eta}} \simeq k(\eta)[[y_1,\ldots,y_r]]$
as $ \sum_{\beta \in \N^r}
c_\beta y^\beta $, where $c_\beta \in k(\eta)$
%\footnote{ \HB{ add power series expansion to setup notation for lemma}}
and set
\begin{equation}\label{eq:qmval}
v_\alpha(f):=\min\{ \langle \alpha, \beta \rangle \, \vert \, c_\beta \neq 0 \}.
\end{equation}
Note that $v_\alpha$ is determined by the Newton polygon of $\sum_{\beta \in \N^r}
c_\beta y^\beta $.

A valuation of the form $v_\alpha$ is called \emph{quasi-monomial}.
If $\alpha \in  \Q_{\geq 0}^r$,
then $v_\alpha$ is a divisorial valuation. Indeed, after a sequence of smooth blowups $Y'\to Y$
that are
toroidal with respect to
the coordinates $y_1,\ldots,y_r$,
we may find a prime divisor $F\subset Y'$  and $c\in \Q_{>0}$ so that $v_\alpha = c\,\ord_F$.

Let $E= E_1+\cdots + E_d$ be a reduced snc divisor on $Y$.
Fix a subset $J\subseteq \{1,\ldots, d\}$ and an irreducible component $Z\subseteq \cap_{i \in J} E_i$. 
Write $\eta \in Y$ for the generic point of $Z$
and choose a regular system of parameters $( y_i)_{i\in J}$ at $\eta$ such that each $y_i$ locally defines $E_i$ at $\eta$.
 We write
${\rm QM}_{\eta}(Y,E)\subseteq \Val_{X}$
for the set of quasi-monomial valuations that can be described at $\eta$ with respect to $(y_i)_{i \in J}$ and note that ${\rm QM}_{\eta}(Y,E)\simeq \R_{\geq 0}^r$.
We set ${\rm QM}(Y,E):= \cup_{\eta} {\rm QM}_\eta(Y,E)$, which has the structure of a simplicial cone complex, and $\QM(Y,E)^*$ for the non-trivial valuations in $\QM(Y,E)$.

\subsubsection{Log discrepancy}
For a pair $(X,\Delta)$, we write 
$$A_{X,\Delta}\colon \Val_X^* \to \R \cup \{ +\infty \} $$ 
for  the log discrepancy function with respect to  $(X,\Delta)$ as in \cite{JM12, BdFFU15}
(see \cite{Blu18} for the case when $\Delta\neq 0 $).
The function $A_{X,\Delta}$ is homogeneous of degree $1$
and lower semicontinous.

A pair $(X,\Delta)$ is klt (resp., lc) if and only if $A_{X,\Delta}(v)>0$ (resp., $\geq0$) for all $v\in \Val_X^*$. 
If $D$ is an effective $\Q$-Cartier divisor, then $A_{X,\Delta+D}(v)= A_{X,\Delta}(v)- v(D)$ for all $v\in \Val_X$.

When $\mu:Y \to X$ is a proper birational morphism with $Y$ normal and $E\subset Y$ a prime divisor, 
\[
A_{X,\Delta}( \ord_E) =  1+ {\rm coeff}_{E}\left( K_{Y}- \mu^*(K_X+\Delta) \right)
\] and we will often write $A_{X,\Delta}(E)$ for this value.
If $\mu:Y\to X$ is a log resolution of $(X,\Delta)$
and $E:=\Exc(\mu)+ \Supp(\mu_*^{-1}\Delta)$, then
$A_{X,\Delta}$ is linear on the cones in $\QM(Y,E)$.
Additionally, if we write 
$\Delta_Y$ for the $\Q$-divisor satisfying $K_Y+\Delta_Y= \mu^*(K_X+\Delta)$, 
then
$A_{X,\Delta}=A_{Y,\Delta_Y}$.

The following result is well known.

\begin{lem}\label{l-lcplacesQM}
Keep the above notation. 
If $(X,\Delta)$ is lc, then
 \[
{\rm QM} (Y,\Delta_Y^{=1})
 =
 \{v \in \Val_{X} \, \vert \, A_{X,\Delta}(v)=0 \}
, \]
where  $\Delta_{Y}^{=1}$  
 is the sum of the prime divisors in $\Delta_Y$ with coefficient one. In particular, the set does not depend on $Y$.
% \footnote{\HB{Note sure if this statement is necessary to prove. We could remove it if people want .}}
\end{lem}
\begin{proof}
Set $E:= \Exc(\mu)+ \Supp(\mu_*^{-1}\Delta)$
and
observe that $A_{X,\Delta}=A_{Y,\Delta_Y}$ is zero on an extremal 
ray of a cone in $\QM(Y,E)$ if and only if 
the corresponding prime divisor on $Y$ has coefficient 1 in $\Delta_Y$. 
Since $A_{X,\Delta}$ is linear on the cones in $\QM(Y,E)$,
this implies
$\QM(Y,\Delta_Y^{=1}) $
is the locus of $\QM(Y,E)$ where  $A_{X,\Delta}$ is zero. 
By \cite[Prop 3.2.5]{Blu18}, it is also the locus of $\Val_X$ where $A_{X,\Delta}$ is zero.
\end{proof}

\subsection{Invariants associated to log Fano pairs}
Let $(X,\Delta)$ be a log Fano pair and $r$ a positive integer such that $L:= -r(K_X+\Delta)$
is a Cartier divisor. The section ring of $L$ is given by
\[
R(X,L) : = R = 
\bigoplus_{m \in \N} R_m
=
\bigoplus_{m \in \N}  H^0(X, \cO_X(mL)).\]

\subsubsection{Filtrations induced by valuations and associated invariants}
For $v\in \Val_X$ and $\lambda \in \R_{>0}$, we set 
\[
\cF_v^\la R_m : = \{ s\in R_m \, \vert \, v(s) \geq \lambda \}. 
\]
If $v = \ord_E$, where $E$ is a divisor over $X$ arising on a proper birational model $\mu:Y\to X$, then 
\[
\cF_v^\la R_m \simeq H^0\left(Y, \cO_Y(m\mu^*L - \lceil \lambda E \rceil ) \right)
.\]

We consider the following invariants
\[
T(v) : = \sup_{m\in \Z_{>0}} T_{mr}(v),\quad  \text{ where } \quad 
T_{mr}(v) := \frac{1}{mr} \sup \{ \lambda \, \vert \, \cF^{\lambda m} R_m \neq 0 \} \]
and
\[
S(v) := \lim_{ m\to \infty}  S_{mr}(v),
\quad \text{ where } \quad
S_{mr}(v) := \int_{0}^ \infty 
\frac{ \dim (\cF^{\lambda m} R_m )}{mr \dim R_m } \, d\lambda
.\]
When the choice of the log Fano pair $(X,\Delta)$ is not clear from context, 
we write $T_{X,\Delta}(v)$ and $S_{X,\Delta}(v)$
for these values. 

Both invariants can be written in terms of the vanishing of $v$ along classes of anti-canonical divisors. 
Specifically, for $m$ divisible by $r$,
$$T_{m} = \max \{ v( \tfrac{1}{m}D) \, \vert \, D \in |-m(K_X+\Delta)|\}.$$
and
$$S_{m}(v)= \max \{ v( D) \, \vert \, D\in |-K_X-\Delta|_{\Q} \text{ is $m$-basis type }\}.$$
Here, following \cite[Def. 0.1]{FO18}, a $\Q$-divisor $D\in |-K_X-\Delta|_{\Q}$ is called \emph{$m$-basis type} 
if there exists a basis $\{s_1,\ldots, s_{N_m} \}$ of $H^0\left(X,\cO_{X}(-m(K_X+\Delta))\right)$
such that 
$$D =\tfrac{1}{mN_m} \left( \{s_1=0\} +\cdots + \{s_{N_m} =0\} \right).$$

The functions $S$ and $T$ are lower semicontinuous on $\Val_X$
\cite[Prop 3.13]{BJ17} and homogeneous of degree 1 \cite[\S 3.2]{BJ17}.
When $E$ is a divisor over $X$ arising on a proper birational model $\mu:Y\to X$,  then
\[
T(\ord_E)=  \sup \{ t \in \R_{>0}\,\vert \,  -\mu^*(K_X+\Delta)-t E \text{ is pseudoeffective } 
\} 
\]
and 
\[
S(\ord_E) : = \frac{1}{(-K_X-\Delta)^n}
\int_0^\infty \vol( -\mu^*(K_X+\Delta) - tE) \, dt 
\]
We will often write $T(E)$ and $S(E)$ for these values.
See \cite[Sect. 3]{BJ17} for further details.

\subsubsection{Behaviour of $S$ and $T$ on a simplicial cone}

Let $\mu:Y\to X$ be a proper birational morphism with $Y$ regular and $E:= \sum_{i=1}^d E_i $ a reduced snc divisor.

\begin{prop}\label{prop:STcontinuousQM}
The functions $S$ and $T$ are continuous on $\QM(Y,E)$. 
\end{prop}

When $X$ is smooth, the result is a special case of \cite[Prop 5.6]{BoJ18}.
 We provide a proof that works for all log Fano pairs.
 
\begin{proof}
 We will only prove the continuity statement for $S$, since the proof for $T$ is similar.
 To proceed, we first show  that for 
  each positive integer $m$ divisible by $r$, the function $S_m$ is continuous on $\QM(Y,E)$. 
  
%Let $Y \to X$ be a proper birational morphism with $Y$  smooth and  $y_1,\ldots,y_r$ local coordinates at a point $\eta \in Y$. 
%For $\alpha \in \R_{\geq 0}^r$, let $v_\alpha \in \Val_X$ denote the quasi-monomial valuation as in Section \ref{ss-qm}.  We will only prove the continuity statement for $S$, since the proof for $T$ is similar.
%\medskip

For a $\Q$-Cartier divisor $D$ on $X$,
write $\varphi_D: \QM(Y,E) \to \R$
for the continuous function sending $v\mapsto v(D)$.
With this notation, we have
$S_m = \sup_{D } \varphi_D$, where the sup runs through all  $m$-basis type divisors. 
Since any $m$-basis type divisor $D$ lies in ${\frac{1}{mN_m} | -mN_m(K_X+\Delta)|}$, Lemma \ref{lem:linearseriesQMfinite} implies the set of functions $\{ \varphi_D\, \vert\, D \text{ is $m$-basis type} \}$ is finite. 
Therefore,
$ S_{m}: \QM(Y,E) \to \R$ is the maximum  of finitely many continuous functions and itself continuous.

We proceed to show $S$ is continuous on $\QM(Y,E)$. 
Since  $S$
is lower semicontinuous on $\Val_X$ \cite[Prop. 3.13]{BJ17}, 
 it suffices to show the upper semicontinuity.
 Pick any $t\in \R_{>0}$. 
We have to show $U: =\{ v \in \QM(Y,E) \, \vert \, S(v) < t \} $
is open. 

Pick any $w \in U$. We may choose $\varepsilon>0$ so that
 $S(w) + \varepsilon A_{X,\Delta}(w)< t$. 
 By the fact that $S_m$ converges pointwise to $S$ and \cite[Thm. 5.13]{BL18},
  which is a partial uniform convergence result, we may choose $m$ divisible by $r$ so that 
 $ S_m(w)+ \varepsilon A_{X,\Delta}(w) < t  $
 and $S \leq S_m + \varepsilon A_{X,\Delta}$ on $\QM(Y,E)$. 
 Since $S_m$ and $A_{X,\Delta}$ are continuous on  $\QM(Y,E)$, 
there exists an open neighborhood $w \in W\subset \QM(Y,E)$ so that 
 $S_m + \varepsilon A_{X,\Delta} < t$ on $W$. 
 Then $W\subseteq U$, which completes the proof.
\end{proof} 

We must prove the following lemma  used in the above proposition. For a $\Q$-Cartier divisor $D$ on $X$, write 
$\varphi_D: \QM(Y,E) \to \R$ for the function sending $v \mapsto v(D)$. 

\begin{lem}\label{lem:linearseriesQMfinite}
If $H$ is a Cartier divisor on $X$, then the set of functions $\{ \varphi_D\, \vert \, D \in | H| \}$
is finite.
\end{lem}

\begin{proof}
It suffices to prove the statement for the restriction of $\varphi_D$ to a fixed simplicial cone in $\QM(Y,E)$. Choose any irreducible component $ Z\subseteq \cap_{i \in J} E_i$. Write $\eta\in Y$ for the generic point of $Z$, set $r: =|J|$,
and fix a regular system of parameters $(y_i)_{i \in J}$
at $\eta\in Y$ such that $y_{i}$ locally defines $E_{i}$. 

Set $B: = \mathbb{P}(H^0(X,\cO_X(H))^*)$
and write $\mathcal{H}$ for the universal divisor on $X\times  B$ parameterizing elements of $|H|$. 
To prove the lemma, we will write $B= \cup B_i$ as a finite  union of constructible subsets so that the restriction of $\varphi_{\mathcal{H}_b}$ to $\QM_{\eta}(Y,E)$ is independent of $b\in B_i$.

Choose  a nonempty affine subset $U \subseteq B$
and a function $f\in \cO_{Y,\eta}\otimes_k \cO(U)$
that defines the Cartier divisor $\mathcal{H}\vert_{Y\times B}$ in a neighborhood of $\eta \times U$.  We can write the image of $f$ in $
\widehat{ \cO_{Y,\eta}} \otimes \cO(U)$
%k(\eta)[[y_i  \vert i \in J]]\otimes \cO(U)$ 
as 
$\sum _{\beta \in \N^r} c_\beta y^\beta$, where each $c_\beta \in k(\eta) \otimes \cO(U)$
and consider the associated Newton polygon $N: = {\rm conv}\{ \beta + \R^r_{\geq 0}\, \vert \, c_\beta \neq 0 \}$. 
Note that $N$ is determined by a finite collection of non-zero coefficients $c_{\beta^{(1)}},\ldots, c_{\beta^{(m)}}$.
Hence, if we let $B_1\subset U$ denote the open set where $c_\beta^{(i)}\neq 0 $ for all $i=1,\ldots, m$, then the Newton polygon of the image of $f$ in $\widehat{\cO_{Y,\eta}}\otimes k(b)$ agrees with $N$ for all $b\in B_1$. Hence, $\varphi_{\mathcal{H}_b}$ is independent of  $b\in B_1$. Repeating this argument on the complement eventually yields such a decomposition.
\end{proof}

\subsubsection{K-stability}
The definition of K-stability was originally defined in terms of degenerations \cite{Tia97,Don02}. 
 For this paper, we use the valuative characterization of K-stability invented in \cite{Fuj19, Li17}, which suits our techniques better.

Let $E$ be a prime divisor over a log Fano pair $(X,\Delta)$
arising on a proper normal model $\mu:Y \to X$.
Following \cite{Fuj19}, we set
\[
\beta_{X,\Delta}(E): = 
A_{X,\Delta}(E)(-K_X-\Delta)^n - 
\int_0^\infty \vol(-\mu^*(K_X-\Delta)-tE) \, dt
.\]

\begin{defn-thm}\label{d-ksemi}
A log Fano pair $(X,\Delta)$ is \emph{K-semistable} (resp., \emph{K-stable})
if and only if $\beta_{X,\Delta}(E)\geq0$ (resp., $>0$)
for all divisors $E$ over $X$.
\end{defn-thm}

The equivalence of this definition with the definition in \cite{Tia97,Don02,LX14} is addressed in 
\cite{Fuj19,Li17} (and \cite{BX18} for part of the K-stable case).

\subsubsection{Thresholds}
Let $(X,\Delta)$ be a log Fano pair and $r$ a positive integer so that $r(K_X+\Delta)$ is Cartier. 
We will consider two thresholds that measure the singularities of anticanonical divisors. 

First is an invariant defined  in \cite{FO18}. 
For a positive integer $m$ divisible by $r$, we set
\[
\delta_{m} (X,\Delta):= \min  \{ \lct(X,\Delta;D) \, \vert \, D\in |-K_X-\Delta|_{\Q} \text{ is $m$-basis type} \}
.\]
The \emph{stability threshold} of $(X,\Delta)$ is defined by $\d(X,\Delta) := \limsup_{m \to \infty} \d_{mr}(X,\Delta)$. 
As shown in  \cite{BJ17},
the
  limsup in the definition of the stability threshold is in fact a limit and 
\begin{equation}\label{deltaVal}
\d(X,\Delta) 
= \inf_{E} \frac{A_{X,\Delta}(E)}{S(E)} 
= \inf_{v} \frac{A_{X,\Delta}(v)}{S(v)}
,\end{equation}
where the first infimum runs through all prime divisors $E$ over $X$ and the second through $v\in \Val_X^*$ with $A_{X,\Delta}(v)<+\infty$. Therefore, the valuative criterion in Definition-Theorem \ref{d-ksemi} implies $(X,\Delta)$ is K-semistable
if and only if $\d(X,\Delta)\geq 1$.

Next is Tian's \emph{$\alpha$-invariant} (also known as the \emph{global log canonical threshold})
defined by 
\[ 
 \alpha(X,\Delta):= \inf \{ \lct(X,\Delta;D) \mid  D\in |-K_X-\Delta|_{\Q} \}
.\]
Similar to the stability threshold, the invariant may be expressed in terms of valuations
\begin{equation}\label{alphaVal}
\alpha(X,\Delta): 
= \inf_{E} \frac{A_{X,\Delta}(E)}{T(E)} 
= \inf_{v} \frac{A_{X,\Delta}(v)}{T(v)};
\end{equation}
see \cite{Amb16,BJ17}.

We say that a valuation computes the stability threshold (resp., global log canonical threshold) if it achieves the infimum  
in \eqref{deltaVal} (resp., \eqref{alphaVal}).

\subsection{Complements}

The theory of complements was introduced by Shokurov in his work on threefold log flips \cite{Sho92}. The boundedness of complements proved in \cite{Bir16a} (also see its generalization in \cite{HLS19}) plays a key role in this paper. 

\begin{defn}[Global complements]
Let $(X,\Delta)$ be a projective lc pair. 
A \emph{$\Q$-complement} of $(X,\Delta)$ is a $\Q$-divisor $\Delta^+$
on $X$ such that $\Delta^+\geq \Delta$,  $(X,\Delta^+)$ is lc, and $K_X+\Delta^+\sim_{\Q} 0$. 
An \emph{$N$-complement} of $(X,\Delta)$ is a $\Q$-complement $\Delta^+$
satisfying 
    $N(K_X+\Delta^+) \sim 0$. 

The latter  definition differs from the terminology in \cite{Bir16a}, which is weaker.
An $N$-complement $\Delta^+$ of $(X,\Delta)$ as defined above agrees with the definition in \emph{loc. cit.} of an $N$-complement $\Delta^+$ of $(X,\Delta)$ satisfying $\Delta^+\geq \Delta$, which is sometimes called a monotonic $N$-complement in the literature.%\footnote{\YL{added monotonic}}
\end{defn}

Clearly, if $\Delta^+$ is an $N$-complement,
then
$\Delta^+ -\Delta \in |-K_X-\Delta|_{\Q}$. %\footnote{\YL{Are we assuming monotonicity?} \HB{According to Definition 2.7, all complements in our paper are monotonic. Should we point out that this is referred to as a monotonic complement in the language of [Bir19]? } \CX{I agree we could mention our terminology is slightly different with Birkar's.} }
Additionally, if $r$ is a positive integer so that  $r(K_X+\Delta)$ is Cartier, then $rN(\Delta^+ - \Delta) \in |-rN(K_X+\Delta)|$.

One crucial input in Theorem \ref{t-constructible} is the following statement, which  follows from the deep result of \cite[Theorem 1.7]{Bir16a}.
\begin{thm}[{\cite[Theorem 1.7]{Bir16a}}]\label{t-complement}
Let $n$ be a natural number and $I\subset \mathbb{Q}\cap [0,1]$ a finite set. There is a positive number $N:=N(n,I)$ depending only on $n$ and $I$ satisfying the following:

Assume $(X,\Delta)$ is an $n$-dimensional lc pair such
that $X$ is of Fano type and the coefficients of $\Delta$ belong to $D(I)$. 
If $(X,\Delta)$ admits a $\bQ$-complement, then it admits an $N$-complement.
\end{thm}

%Here, $D(I) : = \{ 1 - r/m\, \vert \, r\in I, m \in \N \}$. \HB{Added def of $D(I)$ here. Should we use $D(I)$ throughout the paper?}

\begin{proof} Since  $-K_X-\Delta$ is not nef, we cannot directly apply \cite[Thm 1.7]{Bir16a}.
But, there is an an easy reduction step (see e.g. \cite[(6.1)]{Bir16a}). 

Since $X$ is Fano type and $-K_{X}-\Delta$
is linearly equivalent to the effective divisor $\Delta^+-\Delta$, we can run an MMP for $-K_X-\Delta$ to get a birational model $h\colon X\dasharrow X'$ such that $-K_{X'}-h_*\Delta$ is  nef. Since $(X,\Delta)$ has a $\mathbb{Q}$-complement, so does $(X',h_*\Delta)$. 
Now,  $(X', h_*\Delta)$
has an  $N$-complement by \cite[Theorem 1.7]{Bir16a}.
Therefore, \cite[(6.1)]{Bir16a} implies $(X,\Delta)$
admits an $N$-complement as well.
\end{proof}

We note that one can find more general statements on the existence of bounded complements in \cite[Thm 1.13]{HLS19}.

\section{Approximation and boundedness}\label{s-approx+bound}
The idea of approximating a valuation by a sequence of divisors coming from a special type of birational morphisms was developed in \cite{LX16,Fuj17}, modeled on the arguments in \cite{LX14}. One key observation in this paper is that we can combine the boundedness of complements with the latter approximation process. 
%More precisely, in this section, we will show that when the stability threshold is $<1$ it may be approximated by valuations that  are lc places of a bounded family of complements. 

\subsection{Approximation of thresholds and  $\Q$-complements}
We will proceed to discuss an important class of valuations over a log Fano pair
and then describe their relation to the stability  and  global log canonical thresholds.

\subsubsection{Lc places of $\Q$-complements}

\begin{defn}
Let $(X,\Delta)$ be a log Fano pair. 
We say $v\in \Val_X$ is an \emph{lc place of a $\Q$-complement} (resp., $N$-\emph{complement}) if there exists a $\Q$-complement (resp., $N$-complement) $\Delta^+$ of $(X,\Delta)$ such that 
$A_{X,\Delta^+}(v) =0$. When $v=\ord_E$ for some divisor $E$ over $X$, we simply say $E$ is an lc place of a $\Q$-complement (resp., $N$-complement).
\end{defn}
In Appendix \ref{s-weaklyspecial}, we will see that lc places of a $\Q$-complements  are closely related to weakly special test configurations with irreducible central fiber. 

\medskip

%If $\Delta^+$ is a complement of $(X,\Delta)$, the valuations 
%that are lc places of $(X,\Delta^+)$ have a simple description:
%Let$\pi:Y\to X$ be a log resolution of $(X,\Delta^+)$,
%define
%$\Delta^+_Y$ by the equation
% $K_Y+\Delta^+_Y=\pi^*(K_X+\Delta^+)$,
% and write $E= E_1 + \cdots + E_r$ for the sum of the prime divisors in $\Delta^+_Y$ with coefficient 1. 
% We claim that $\QM(Y,E)$ is the set of valuations with $A_{X,\Delta^+}=0$.
%To see this, note that
% $A_{X,\Delta}$ is
% linear on the cones in $\QM(Y,E)$ and takes the value zero on the 
% extremal rays (since they correspond to the valuations
%$\ord_{E_i}$)).
%Hence, $A_{X,\Delta}$ is zero on $\QM(Y,E)$. 
%The reverse inclusion follows from \cite[Prop. 3.2.5]{blu18}). 

We state the following elementary lemma concerning divisorial valuations that are the lc place of a $\Q$-complement.

\begin{lem}\label{l-lcplace}
Let $(X,\Delta)$ be a log Fano pair and $E$ a prime divisor over $X$.
 If $E$ is an lc place of a $\Q$-complement,
then there exists a proper birational morphism of normal varieties
$\mu:Y\to X$ satisfying:

\begin{enumerate}
\item $E$ appears as a divisor on $Y$ with $\Exc(\mu) \subseteq E$,
%such that $-E$ is $\Q$-Cartier and $\mu$-ample, (removed, since not needed for application of Lemma in Thm 3.5 -HB
\item $(Y,  \mu_*^{-1}\Delta +(1-a)E )$ is lc and admits a $\Q$-complement, where 
 $a: ={\rm coeff}_{E}(\Delta)$ if $E$ is a prime divisor on $X$ and zero otherwise, and
\item $Y$ is Fano type.
\end{enumerate}
\end{lem}

%Recall, a projective variety $Y$ is \emph{Fano type} if there exists a $\Q$-divisor $B$ such that $(Y,B)$ is klt  and $-K_Y-B$ is big and nef.

\begin{proof}
Choose a $\Q$-complement $\Delta^+$ of $(X,\Delta)$ such that $A_{X,\Delta^+}(E)=0$ and set 
$$D:= \Delta^+-\Delta\sim_{\bQ}-(K_X+\Delta).$$ 
Fix $0<c<1$ so that $0< A_{X,\Delta+cD}(E)<1$.

By \cite{BCHM10}, there exists a proper birational morphism of normal varieties $\mu:Y \to X$ such that $E$ appears as a divisor on $Y$ with $\Exc(\mu) \subseteq E$ such that   $-E$  is $\Q$-Cartier and $\mu$-ample.
Write $\Gamma$ for the $\Q$-divisor on $Y$ so that $$K_{Y}+\Gamma=\mu^*(K_X+\Delta^+).$$
Since  $(Y,\Gamma)$ is lc and
$\Gamma \geq \mu_*^{-1}\Delta+ (1-a)E$, (2) holds.
Next, set $\Gamma' = \Gamma-(1-c)\mu^*(D)$. 
Note that $(Y,\Gamma')$ is  klt, since $(X,\Delta+cD)$ is klt. Additionally,
$$
-K_Y -\Gamma'\sim_{\Q}  
-(1-c) \mu^*(K_X+\Delta)
$$
is big and nef. Therefore, $Y$ is of Fano type.
\end{proof}

\subsubsection{Stability threshold}

We state the following characterization of the stability threshold. 

\begin{prop}\label{p-blz}
Let $(X,\Delta)$ be a log Fano pair. 
If $\delta(X,\Delta)\le 1$, then 
\[
\delta(X,\Delta) = \inf_{E} \frac{A_{X,\Delta}(E)}{S(E)}
\]
where the infimum runs through prime divisors $E$ over $X$ 
such that $E$ is an lc place of a $\Q$-complement.
\end{prop}

The result is proved in \cite{BLZ19} in the case when $\delta(X,\Delta)<1$. %using the characterization of the stability threshold in terms of $m$-basis type divisors \cite{FO18}% and a uniform approximation result in \cite{BJ17} removed since,
Using an argument from \cite{ZZ19}, the $\delta(X,\Delta)=1$ case can be deduced from the $<1$ case. 

\begin{proof}
We first treat the case when $\delta(X,\Delta)<1$ which is embedded in the proof of \cite[Theorem 4.1]{BLZ19}. For the convenience of the reader, we recall the argument in \emph{loc. cit.}

By Equation \ref{deltaVal}, the inequality 
$\delta(X,\Delta) \leq \inf_E \tfrac{A_{X,\Delta}(E)}{S(E)}$ holds. 
%Next, set $\delta:=\delta(X,\Delta)$ and $\delta_m:=\delta_m(X,\Delta)$. 
%By \cite[Proof of Theorem 4.1]{BLZ19}, for each  positive integer $m$ sufficiently divisible there exists a prime divisor $E_m$ over $X$ and an $m$-basis type divisor $B_m\in |-K_X-\Delta|_{\bQ}$ such that each $E_m$ is an lc place of the log canonical pair $(X,\Delta+\delta_m B_m)$, $\delta_m<1$, and  $\delta(X,\Delta)=\lim_{m\to\infty}\frac{A_{X,\Delta}(E_m)}{S(E_m)}$. 
%\footnote{
%\HB{Is it worth repeating the argument from BLZ19 for the convenience of the reader, since it is short:\\
For the reverse inequality, pick any $\varepsilon>0$. By the fact that $\delta$ is a limit and \cite[Cor. 3.6]{BJ17}, we may choose $m$ so that $$\delta_m(X,\Delta)< \min \{1, (1+ \varepsilon) \delta(X,\Delta) \}$$ and $S_m(v)\leq (1+\varepsilon)S(v)$ for all $v\in \Val_X^*$ with finite log discrepancy. Now, fix an $m$-basis type divisor $B$ such that $\delta_m (X,\Delta)= \lct(X,\Delta;B)$
and a divisor $E$ over $X$ computing the lct (i.e. $\tfrac{A_{X,\Delta}(E)}{\ord_E(B)} = \lct(X,\Delta;B)$). Note that $\ord_E(B)\leq S_m(E) \leq (1+\varepsilon ) S(E)$. Therefore, 
\[
\frac{A_{X,\Delta}(E)}{S(E)}\leq (1+\varepsilon) \frac{A_{X,\Delta}(E)}{\ord_E(B)} =
(1+\varepsilon)\delta_m(X,\Delta)
\leq (1+\varepsilon)^2 \delta(X,\Delta).\]

We will show $E$ is the lc place of a $\Q$-complement.
To proceed, note that $(X,\Delta+ \delta_m B)$ is  lc and $A_{X,\Delta+ \delta_m B}(E)=0$. By \cite[Lem. 5.17.2]{KM98}, we can choose a 
divisor $H\in |-K_{X}-\Delta|_{\Q}$ so that $(X,\Delta +\delta_m B+ (1-\delta_m)H)$ remains lc.
Hence,  ${\Delta^+:= \Delta +\delta_m B+ (1-\delta_m)H}$ is a $\bQ$-complement of $(X,\Delta)$ and $A_{X,\Delta^+}(E)=0$.
Therefore, sending $\varepsilon\to 0$ shows that the reverse inequality  holds. 
% }}
%By \cite[Lem. 5.17.2]{KM98}, we can choose a  general divisor $H_m\in |-K_X-\Delta|_{\bQ}$ such that $(X,\Delta+ \delta_m B_m+(1-\delta_m)H_m)$ remains lc. Hence, $\Delta^+_m:=\Delta+ \delta_m B_m+(1-\delta_m)H_m$ is a complement of $(X,\Delta)$ with lc place along $E_m$. 
%Thus, the statement holds when $\delta(X,\Delta)<1$. 

We now assume $\delta(X,\Delta)=1$. 
\medskip

\emph{Claim}: For any $\varepsilon \in  (0, \alpha(X,\Delta))$,
there exists $D\in |-K_{X}-\Delta|_{\Q}$
such that $(X,\Delta+ \varepsilon D)$ is klt and
$\d(X,\Delta+\varepsilon D)<1$.

The claim follows immediately from \cite[Theorem 1.2]{ZZ19}. For the convenience of the reader, we recall the argument in \emph{loc. cit.}
Pick any $\varepsilon \in  (0, \alpha(X,\Delta))$, then $(X,\Delta+\varepsilon D)$ is klt for any $D\in |-K_{X}-\Delta|_{\Q}$.
Since $\d(X,\Delta)=1$, we may choose  a prime divisor $E$ over $X$ such that $\tfrac{A_{X,\Delta}(E)}{S_{X,\Delta}(E)}<1+\tfrac{\varepsilon}{3n}$, where $n:=\dim (X)$. 
Using the inequality $T_{X,\Delta}(E) \geq(1 +\tfrac{1}{n})S_{X,\Delta}(E)$
\cite[Prop. 2.1]{Fuj19}
, we may choose $D\in |-K_X-\Delta|_{\Q}$
such that $\ord_E(D) \geq (1 + \tfrac{1}{2n}) S_{X,\Delta}(E)$. 
Observe that
$$
A_{X,\Delta+\varepsilon D}(E)  = A_{X,\Delta}(E) - \varepsilon\cdot \ord_E(D)
\quad
\text{ and }
\quad
S_{X,\Delta+\varepsilon D}(E)=(1-\varepsilon)S_{X,\Delta}(E),$$ 
where the second equation is \cite[Lem. 3.7.i]{BJ17}. Therefore, 
\[
\delta(X,\Delta+\varepsilon D) 
\leq \frac{A_{X,\Delta+\varepsilon D}(E)}{
S_{X,\Delta+\varepsilon D}(E)}
=
\frac{1}{1-\varepsilon}
\left( 
\frac{A_{X,\Delta}(E)}{S_{X,\Delta}(E)}
- \varepsilon\cdot\frac{\ord_E(D)}{S_{X,\Delta}(E)}
\right) 
.\]
Since 
$\tfrac{A_{X,\Delta}(E)}{S_{X,\Delta}(E)}
- \varepsilon\cdot\tfrac{\ord_E(D)}{S_{X,\Delta}(E)}
\leq 1+\tfrac{\varepsilon}{3n} - \varepsilon \left( 1+\frac{1}{2n}\right)<1-\varepsilon$, we can conclude $\delta(X,\Delta+\varepsilon D) <1$. 
\medskip 

We now return to the proof of the proposition.
By Equation \ref{deltaVal}, the inequality 
$\delta(X,\Delta) \leq \inf_E \tfrac{A_{X,\Delta}(E)}{S(E)}$ holds. 
For the reverse inequality, fix a rational number $\varepsilon \in (0,\alpha(X,\Delta))$. 
By the above claim, there exists $D\in |-K_X-\Delta|_{\Q}$ such that $\delta(X,\Delta+\varepsilon D)<1$. 
Using the $\delta<1$ case, we may find a divisor $E$ over $X$ such that $\tfrac{A_{X,\Delta+\varepsilon D}(E)}{S_{X,\Delta+\varepsilon D}(E)}<1
$
and $E$ is the lc place of a $\mathbb Q$-complement $\Delta^+$
of $(X,\Delta+\varepsilon D)$. 
Note that $\Delta^+$ is also a $\Q$-complement of $(X,\Delta)$, since by definition $\Delta^+\ge \Delta+\varepsilon D\ge \Delta $.%\footnote{\CX{I add some simple explanation to possibly clarify the referee Y's confusion.}}

We now estimate $\tfrac{A_{X,\Delta}(E)}{S_{X,\Delta}(E)}$.
Using that $\alpha(X,\Delta) \leq
\tfrac{A_{X,\Delta}(E)}{T_{X,\Delta}(E)}$ by \eqref{alphaVal}, we see 
$\ord_E(D) \leq T_{X,\Delta}(E)\leq  \tfrac{A_{X,\Delta}(E)}{ \alpha(X,\Delta)}$. 
Therefore,
\[
1>\frac{A_{X,\Delta+\varepsilon D}(E)}{
S_{X,\Delta+\varepsilon D}(E)}
= 
\frac{A_{X,\Delta}(E)- \varepsilon \ord_E(D)}
{(1-\varepsilon)(S_{X,\Delta}(E))}
\geq 
\left( \frac{1- \varepsilon/ \alpha(X,\Delta)}{1-\varepsilon}
\right) 
\frac{A_{X,\Delta}(E)}{S_{X,\Delta}(E)}
.\]
Sending $\varepsilon \to 0$ completes the proof. 
\end{proof}

\subsubsection{Global log canonical threshold} 

We now prove an analog of Proposition \ref{p-blz} for the global log canonical threshold.
The statement follows almost immediately from  definitions.

\begin{prop}\label{p-alphaQ-comp}
Let $(X,\Delta)$ be a log Fano pair. 
If $\alpha(X,\Delta)<1$, then 
\[
\alpha(X,\Delta) = \inf_{E} \frac{A_{X,\Delta}(E)}{T(E)}
\]
where the infimum runs through divisors $E$ over $X$ 
such that $E$ is an lc place of $\Q$-complement.
\end{prop}

By applying a deeper result \cite[Thm. 1.5]{Bir16b}, it follows that the above infimum is a minimum.
Though, Proposition \ref{p-alphaQ-comp}
will be sufficient for proving that $\min\{1,\alpha\}$ is constructible.

\begin{proof}
By Equation \ref{alphaVal}, $\alpha(X,\Delta)\leq \inf_{E} \frac{A_{X,\Delta}(E)}{T(E)}$ where the infimum runs through lc places of $\Q$-complements. 
For the reverse inequality, pick any $\varepsilon \in (0, 1-\alpha(X,\Delta))$.
We may choose $D \in |-K_{X}-\Delta|_{\Q}$ such that 
$$c: = \lct(X,\Delta;D) \leq \alpha(X,\Delta)+\varepsilon,$$
and a divisor $E$ over $X$ computing $\lct(X,\Delta;D)$. Note that $c<1$ by our choices of $\varepsilon$.

Observe that $(X,\Delta+cD)$ is lc 
and $A_{X,\Delta+cD}(E) =A_{X,\Delta}(E)- c \cdot \ord_E(D)=0$.
Since $-K_X-\Delta$ is ample, we may find $H \in |-K_X-\Delta|_{\Q}$
so that $(X,\Delta+cD+(1-c)H)$ remains lc \cite[Lem. 5.17.2]{KM98}.
Hence, $\Delta^+:= \Delta+cD+(1-c)H$ is a $\Q$-complement of $(X,\Delta)$ 
with $A_{X,\Delta^+}(E)=0$.  

Now, observe that
$\frac{A_{X,\Delta}(E)}{T(E)}
\leq \lct(X,\Delta;D)
 \leq \alpha(X,\Delta)+\varepsilon,$
 since $\ord_E(D)\leq  T(E)$. Therefore, sending $\varepsilon\to 0$ completes the proof. 
\end{proof}

\subsection{Boundedness} 
Using the boundedness of complements, we will show that lc places of $\Q$-complements are in fact lc places 
of $N$-complements.

\begin{thm}\label{t-globcomp}
Let $n$ be a natural number and $I\subseteq \Q$ a finite set.
There is a positive integer $N:= N(n,I)$
satisfying the following: 

Assume $(X,\Delta)$ is an $n$-dimensional log Fano pair such that the coefficients of $\Delta$ belong to $D(I)$.
If $E$ is a divisor over $X$ 
that is the lc place of a $\Q$-complement, then $E$ is an lc place of an $N$-complement.
\end{thm}

The statement is a consequence of the boundedness of complements in \cite{Bir16a}. 

\begin{proof}
Let $E$ be a divisor over a log Fano pair $(X,\Delta)$ such that 
$E$ is an lc place of a $\Q$-complement.
Applying Lemma \ref{l-lcplace} gives  a proper birational morphism $\mu:Y\to X$ satisfying conditions (1)-(3) of the lemma. 

Since the latter conditions are satisfied, we may apply Theorem \ref{t-complement} to find an integer $N:=N(n,I)$, depending only on $n$ and $I$, so that $(Y,  \mu_*^{-1}\Delta+(1-a)E  )$ admits a $N$-complement $\Gamma_Y$.

Hence, $N(K_Y+\Gamma_Y)\sim 0$,
$(Y,\Gamma_Y)$ is lc, and $\Gamma_Y$ has coefficient 1 along $E$.
If we set $\Gamma:= \mu_*(\Gamma_Y)$,  then 
\[
K_Y + \Gamma_Y = \mu^*(K_X+\Gamma). 
\]
holds.
This implies 
$\Gamma$ is a $N$-complement of $(X,\Delta)$ and $A_{X,\Gamma}(E)=0$. 
\end{proof}

 The next two statements follow immediately
from combining  
Theorem \ref{t-globcomp} with
Propositions \ref{p-blz}
and 
 \ref{p-alphaQ-comp}.

\begin{cor}\label{c-Ncompapproxd}
Let $n$ be a natural number and $I\subseteq \Q$ a finite set.
There is a positive integer $N:= N(n,I)$
satisfying the following: 

If $(X,\Delta)$ is an $n$-dimensional log Fano pair
such that the coefficients of $\Delta$ belong to $D(I)$ and $\delta(X,\Delta)\le 1$, then  
\[
   \delta(X,\Delta) = \inf_{E} \frac{A_{X,\Delta}(E)}{S(E)}
,    \]
 where  the infimum runs through divisors over $X$ that are lc places of an $N$-complement.
 \end{cor}

\begin{cor}\label{c-Ncompapproxa} 
 Let $n$ be a natural number and $I\subseteq \Q$ a finite set.
There is a positive integer $N:= N(n,I)$
satisfying the following: 

If $(X,\Delta)$ is an $n$-dimensional log Fano pair
such that the coefficients of $\Delta$ belong to $D(I)$ and $\alpha(X,\Delta)<1$, then  
\[
   \alpha(X,\Delta) = \inf_{E} \frac{A_{X,\Delta}(E)}{T(E)}
,    \]
 where  the infimum runs through divisors over $X$ that are lc places of an $N$-complement.
\end{cor}
 
 \subsection{Approximating valuations computing the stability threshold}

In this section, we show that if a valuation  computes $\delta<1$, then it is a limit of  divisorial valuations that are lc places of bounded complements. The result will not be used in the proof of Theorem \ref{t-constructible}.
%$\lim_{k\to \infty} \frac{A_{X,\Delta}(v_k)}{S(v_k)} = \delta(X,\Delta)$. In Appendix \ref{ss-compute}

We will obtain stronger results via passing to the cone to also cover the case when $\delta(X,\Delta)=1$. Nevertheless, Proposition \ref{p-minapprox} is  proved only using global arguments. 

\begin{prop}\label{p-minapprox}
Let $n$ be a natural number and $I\subseteq \Q$ a finite set.
There is a positive integer $N:= N(n,I)$
satisfying the following: 

Assume $(X,\Delta)$ is an $n$-dimensional log Fano pair such that the coefficients of $\Delta$ belong to $D(I)$ and  $\delta(X,\Delta)<1$.
If $v^{*}\in \Val_X^{=1}$ computes $\d(X,\Delta)$, then there exists a sequence of divisorial valuations $(v_k)_k$
in $\Val_X^{=1}$ converging  to $v^*$ 
such that each $v_k$ is the lc place of a $N$-complement
and 
$\lim_{k \to \infty}\frac{A_{X,\Delta}(v_k)}{S(v_k)} = \delta(X,\Delta)$.
%\footnote{ 
%\HB{Added $\lim_{k \to \infty}\frac{A_{X,\Delta}(v_k)}{S(v_k)} = \delta(X,\Delta)$ as Chenyang suggested in the paragraph before the proof. I mistakenly thought this didn't follow from the proof. But the implication is trivial and I was confusing inequalities.}  }
\end{prop}

Here, $\Val_{X}^{=1}$ denotes the set  $\{v\in \Val_X\, \vert \, A_{X,\Delta}(v)=1\}$.
Before proving the proposition, we need the following lemma, which may be viewed as a global analogue of \cite[Lemma 3.8]{LX16}.

\begin{lem}\label{l-minapproxideal}
Let $(X,\Delta)$ be a log Fano pair with  $\delta(X,\Delta)<1$.
Assume $v^{*}\in \Val_X^{=1}$ computes $\d(X,\Delta)$. 
For any ideal $\fb$ on $X$  and $\varepsilon >0$, 
there exists a divisorial valuation  $w \in {\rm Val}_X^{=1}$ such that 
\begin{enumerate}
    \item $w$ is the lc place of a $\Q$-complement and
     \item $w(\fb) \geq v^{*}(\fb)( 1 - \varepsilon)$.
\end{enumerate}
\end{lem}

\begin{proof}
Let $\mu:Y \to X$ be a log resolution of $(X,\Delta,\fb )$. 
Write $E$ for the Cartier divisor on $Y$ such that 
$\fb \cdot \cO_Y = \cO_Y(-E)$
and $\Delta_Y$ for the $\Q$-divisor on $Y$ such that
$$K_Y+ \Delta_Y = \mu^*(K_X+\Delta).$$ 
%Since $(X,\Delta)$ is klt, there exists  $a>0$ so that $\Delta_Y \leq (1-a) \Delta_Y^{\red}$. 
Since $-K_X-\Delta$ is ample and $-E$ $\mu$-semiample,
there exists a rational number $0<t\ll 1$ so that $-\mu^*(K_X+\Delta)-tE$ is semiample.

\medskip

\noindent{\emph{Claim}}:
For any $\varepsilon'>0$, there exists
$D \in |-K_X-\Delta|_{\Q}$ such that 
\begin{equation}\label{e-w(D)}
tw( \fb)  \leq w(D) \leq tw(\fb) + \varepsilon' A_{X,\Delta}(w)
\end{equation}
for all $w\in \Val_X$ and $\lct(X,\Delta;D) >\lct(X,\Delta;\fb)/2$.
\medskip

Fix $m\in \Z_{>0}$ sufficiently divisible so that 
$$|m\left(-\mu^*(K_X+\Delta)-t E\right)|$$
is base point free.
By Bertini's Theorem, we may choose a divisor $H$ in the above linear system
so that  $\Supp(\Delta_Y)+H$ is snc.
Set $D : =\mu_*(m^{-1}H+tE)$, which is an element of $|-K_X-\Delta|_{\Q}$. We will show $D$ satisfies the claim if $m\gg0$. 

Note that $\mu^*D = m^{-1}H +tE$. Hence, for $w \in \Val_X$, 
\[
w(D) = m^{-1}w(H) + tw(E) = m^{-1}w(H) +tw(\fb).
\]
Since $H +\Supp(\Delta_Y)$ is snc, $\lct(Y, \Delta_Y;H)=1$. Therefore, $A_{X,\Delta}(w) = A_{Y,\Delta_Y}(w) \geq   w(H)$, which implies \eqref{e-w(D)} holds if $m\geq 1/\varepsilon' $. 

To finish the claim, we compute 
\begin{align*}
\lct(X,\Delta;D)^{-1} &=  \lct\left(Y,\Delta_Y; m^{-1}H +tE\right) ^{-1}\\
& \leq   \lct\left(Y,\Delta_Y; m^{-1}H\right) ^{-1} + \lct\left(Y, \Delta_Y; tE\right)^{-1}\\
&  \leq  1/m +  t \cdot \lct\left(X,\Delta; \fb \right)^{-1},
\end{align*}
where the first inequality follows from the log concavity of the log canonical threshold (e.g. see \cite[Lemma 1.7.iv]{JM12}). 
Since $t<1$, we conclude $\lct(X,\Delta;D) >  \lct(X,\Delta; \fb)/2$ when $m\gg0$. 

\medskip
Returning to the proof of the lemma,
set $\delta:=\d(X,\Delta)$, 
$c: = \lct(X,\Delta; \fb)$,
and  $\beta := \min\{1-\delta, c/2\}$. 
Choose a divisor $D$ satisfying the above claim with $\varepsilon': = \varepsilon t v^* (\fb) /2$. 
If we set $\Delta': = \Delta+\beta D$,
then $(X,\Delta')$ is log Fano. 
Indeed, the pair is klt, since $\lct(X,\Delta;D)>c/2 \geq \beta$. Additionally,
$-K_X-\Delta'\sim_{\Q} -(1-\beta)(K_X+\Delta)$ is ample. 

Observe that for $w\in \Val_{X}$, 
$$A_{X,\Delta'}(w)= A_{X,\Delta}(w)- \beta w(D)
\quad \text{ and }  S_{X,\Delta'}(w) = (1-\beta)S_{X,\Delta}(w),$$
where the  second equality is by \cite[Lem. 3.7.i]{BJ17}.
Therefore,  \eqref{e-w(D)} gives
\begin{equation}\label{e-logdineq}
A_{X,\Delta}(w)(1-\beta \varepsilon')- \beta t w(\fb)  
\leq 
A_{X,\Delta'}(w)
\leq
A_{X,\Delta}(w)- \beta t w(\fb).
\end{equation}
If we set $p:=v^{*}(\fb)$, we see
\[
\d' := \delta(X,\Delta') 
\leq 
\frac{A_{X,\Delta'}(v^{*}) }
{S_{X,\Delta'}(v^{*})}
\leq 
\frac{ A_{X,\Delta}(v^{*}) - \beta t p}{(1-\beta)S_{X,\Delta}(v^{*})}
= 
\frac{ (1 - \beta tp)\delta }
{1-\beta} 
.\]
Since $1-\beta \geq \d$, it follows that $\d'<1$.

Applying Proposition \ref{p-blz} to $(X,\Delta')$, 
we may find a divisorial valuation $w\in{\rm Val}_{X}^{=1}$ that is an lc place of a $\Q$-complement of $(X,\Delta')$
so that 
\begin{equation}\label{e-epsbeta1}
\frac{A_{X,\Delta'}(w)} 
{S_{X,\Delta'}(w)} \leq 
\frac{ (1 - \beta tp+\beta \varepsilon ') \delta }
{1-\beta} 
.\end{equation}
Since $\Delta'\geq \Delta$, $w$ is also 
 an lc place of a $\Q$-complement of $(X,\Delta)$.
By \eqref{e-logdineq} and the inequality $\d \leq \frac{A_{X,\Delta}(w)}{S_{X,\Delta}(w)} = \frac{1}{ S_{X,\Delta}(w)}$, we know 
\begin{equation}\label{e-epsbeta2}
\frac{  \left(1-\beta \varepsilon ' -  \beta tw(\fb) \right) \delta}{1-\beta }
\leq 
\frac{A_{X,\Delta'}(w)}
{S_{X,\Delta'}(w)}
.
\end{equation}
Analyzing \eqref{e-epsbeta1} and \eqref{e-epsbeta2}, we see $ 
 p \leq  w(\fb) +\frac{2 \varepsilon '}{ t}$. Since $\frac{2\varepsilon'}{ t} = \varepsilon v^*(\fb)$, the proof is complete. 
 \end{proof}

\begin{proof}[Proof of Proposition \ref{p-minapprox}]
By Lemma \ref{l-minapproxideal}, there exists a sequence of divisorial valuations $(v_k)_{k}$ 
in $\Val_X^{=1}$ that are lc places of $\Q$-complements satisfying 
$$\frac{v_k(\fa_k(v^{*}))}{k} \geq 1 - \frac{1}{k}\quad \text{  and } \quad A_{X,\Delta}(v_k) =1.$$  
By Theorem \ref{t-globcomp}, each $v_k$ is also the lc place of a $N$-complement.
We will show 
that a subsequence converges to $v^{*}$ in the valuation space. 

Let $\xi \in X$ denote the center of $v^{*}$.
Write $\mathfrak{m}_{\xi} \subset \cO_X$ 
for the ideal of functions vanishing along $\xi$
and set $r: =v^{*}(\mathfrak{m}_{\xi})$, which is $>0$.
Since $v^{*}(\mathfrak{m}_{\xi}^{ \lceil k/r \rceil }) = \lceil k/r\rceil r \geq k$, 
$
\mathfrak{m}_\xi^{ \lceil  k/r\rceil } \subseteq \fa_k(v^{*})$.
Therefore, 
\[
\lceil  k/ r\rceil v_k( \mathfrak{m}_{\xi}) 
\geq v_k (\fa_k(v^{*} )) 
\geq k-1
\]
for all $k$.
This implies there exists $\varepsilon>0$ so that $v_k(\fm_{\xi}) \geq \varepsilon $ for all $k>1$. 
Since 
\[
c: = \lct( X,\Delta; \mathfrak{m}_{\xi}) = \inf_{w\in \Val_X^*} \frac{A_{X,\Delta}(w)}{w(\mathfrak{m}_{\xi})},\]
we also have 
$v_k(\fm_{\xi}) \leq A_{X,\Delta}(v_k) /c= 1/c$ for all $k$.

Observe that 
\[
V: = \{ v\in \Val_{X}, \vert \, 
 v(\fm_{\xi}) \in [\varepsilon, 1/c] \text{ and } A_{X,\Delta}(v)
\leq  1 \} \subseteq \Val_X
\]
is a compact subset of the valuation space  by \cite[Thm 3.1]{BdFFU15}. Since any compact subset of the valuation space is also sequentially compact \cite{Poi13} (see also the proof of \cite[Prop 3.9]{LX16}),
there exists a subsequence 
$(v_{k_j})_j$
 so that the limit $w^* = \lim_{j \to \infty}v_{k_j}$ exists in $\Val_{X}$.
We will proceed to show $w^*=v^{*}$.

By the lower semicontinuity of the log discrepancy function,
$A_{X,\Delta}(w^*) \leq 1$.
For any positive integers $k$, there is an inclusion $\fa_{m}(v^{*})^{\lceil  k/m \rceil} \subseteq \fa_{k}(v^{*})$. 
This implies
$ 
\lceil k/m\rceil
v_{k}(\fa_m(v^{*})) 
\geq
  v_{k}(\fa_{k}(v^{*}))
\geq
k-1
$. Therefore, 
\[
w^* (\fa_m(v^{*})) = 
\lim_{j \to \infty} v_{k_j}(\fa_m(v^{*}))
\geq 
\lim_{j \to \infty}
\frac{k_j-1}{\lceil k_j/m \rceil}
=
m.\]
From the latter inequality, we see  $w^*(\fa_\bullet(v^{*})) \geq1$.
Hence, $w^* \geq v^{*}$ holds by \cite[Lem. 2.4]{JM12}.
Therefore,   $S(w^*) \geq S(v^{*})$
and equality holds if and only if $w^* = v^{*}$ by  \cite[Prop. 3.15]{BJ17}.

Since $v^{*}$ computes the stability threshold,  $\frac{A_{X,\Delta}(v^{*})}{S(v^{*})} \leq \frac{A_{X,\Delta}(w^*)}{S(w^*)}$. 
Using that $A_{X,\Delta}(w^*)\leq 1 = A_{X,\Delta}(v^{*})$, 
we conclude $S(w^*) =  S(v^{*})$. 
Hence, $w^* = v^{*}$ and $v^* = \lim_{j \to \infty} v_{k_j}$.

Since $S$ is lower semicontinuous on the valuation space $\liminf_{j\to \infty} S(v_{k_j}) \geq  S(v^*)$. Hence, $\limsup_{j\to \infty} \frac{A_{X,\Delta}(v_{k_j})}{S(v_{k_j}) } \leq \frac{A_{X,\Delta}(v^*)}{S(v^*)} = \delta(X,\Delta)$. By  \eqref{deltaVal}, the limit exists and equals $\delta(X,\Delta)$.
%\footnote{ \HB{Added this paragraph}}
\end{proof}

\begin{comment} % \delta=1

\begin{prop}\label{p-minapprox2}
Let $(X,\Delta)$ be a log Fano pair with  $\delta(X,\Delta)\le 1$.
If $v^{ *}\in \Val_X^{=1}$ computes $\d(X,\Delta)$, then there exists a sequence of valuations $(v_k)_{k \in \N}$ in ${\rm DivVal}_X^{=1}$ such that $\lim_{k \to \infty}{v_k} = v $
and each $v_k$ is an lc place of a $\Q$-complement.
\end{prop}
\begin{proof}We may assume $\delta(X,\Delta)=0$. Consider $v$ which computes $\delta(X,\Delta)$. Then for any $\epsilon>0$, we can find a sufficiently large $m_0$, such that for any $m_0 \vert m$, we can find a divisor $m$-basis type divisor $M$, and a divisor $D\sim_{\bQ}-K_X-\Delta$, such that the log discrepancy of $v$ with respect to $(X,\Delta+(1-\epsilon)M+\epsilon D)$ is at most 0. Replace $\epsilon$ by $\epsilon$
\end{proof}
\end{comment}

\section{Constructibility}

\subsection{Invariance of volumes}
To prove Theorem \ref{t-constructible}, 
we will need a constructibility result 
for the functions $S$ and $T$ when the valuation varies in a family.

\medskip
Consider the following setup:
Let $(X,\Delta)\to B$ be a $\Q$-Gorenstein family of log Fano pairs with $B$ smooth.
Let
 $D$ be an effective $\Q$-divisor such that 
$D\sim_{B,\Q} -K_{X/B}-\Delta$,
$\Supp(D)$ does not contain a fiber, 
and  $(X_b,\Delta_b+D_b)$ is lc for all $b\in B$. 

\begin{prop}\label{p-invarianceST}
If $(X,\Delta+D)\to B$ admits a fiberwise log
resolution $g:Y\to X$
and $F$ is a toroidal divisor with respect to $\Exc(g)+\Supp(g_*^{-1} \Delta)$ satisfying
$A_{X,\Delta+ D}(F)<1$, then 
\[
 S_{X_b,\Delta_b}(F_b)
\quad 
\text{ and }
\quad
T_{X_b,\Delta_b}(F_b)\]
 are independent of $b\in B$.
\end{prop}

The \emph{toroidal} condition  in the above theorem  means  $F$ is an exceptional divisor of a sequence of toroidal blowups of $Y$ with respect to $\Exc(g)+\Supp(g_*^{-1} \Delta)$. This is equivalent to the condition that $\ord_F \in \QM(Y, \Exc(g)+\Supp(g_*^{-1} \Delta))$.

The result is a consequence of the deformation invariance of log plurigenera in smooth families \cite[Thm. 1.8]{HMX13}, whose proof is based on comparing the relative MMP over $B$ and the MMP for individual fibers.

\begin{proof}
By shrinking $B$,
we may assume $B$ is affine. 
By repeatedly blowing up the center of $F$ on $Y$,
we may assume $F$ is a prime divisor on $Y$. We fix $t\in \Q_{>0}$ 
and aim to show
\begin{equation}\label{e-volind}
\vol(-g_b^*(K_{X_b}+\Delta_b)-t F_b) 
\end{equation}
is independent of $b\in B$.

Let $\Gamma_1$ and $\Gamma_2$ be the effective $\Q$-divisors without common components in their support such that
$$
K_Y+ \Gamma_1= g^*(K_X+\Delta +D) +\Gamma_2
$$
and $g_* \Gamma_1  = \Delta+D$.
Note that $\Supp(\Gamma_1+\Gamma_2)$ is relative snc over $B$
and $d: = {\rm coeff}_{F}(\Gamma_1) = 1-A_{X,\Delta+D}(F)>0$. By inversion of adjunction we know that $(X,\Delta+D)$ is log canonical. In particular, $\Gamma_1$ has coefficients in $[0,1]$.

Since $-K_{X/B}-\Delta$ is $f$-ample, 
we may use Bertini's Theorem to find an effective $\Q$-divisor 
$H\sim_{B,\Q}-(d/t)(K_{X/B}+\Delta)$
such that $\Gamma_1 + g^*H - dF$ has coefficients in the interval $[0,1]$
and $\Supp(\Gamma_1+ g^*H- dF)$ is relative snc over $B$ after possibly shrinking $B$. 
Applying \cite[Thm. 1.8 (3)]{HMX13} gives that
\begin{equation}\label{e-volinv}
\vol\left(K_{Y_b}+(\Gamma_{1})_b+g_b^*H_b- dF_b\right)
\end{equation}
is independent of $b\in B$.
Observe  that
\begin{align*}
K_{Y}+ \Gamma_1 + g^*H- dF
&\sim_{B,\Q}
g^*(K_{X/B}+\Delta+D+H) -dF+\Gamma_2 \\
&\sim_{B,\Q} 
 -(d/t) g^*(K_{X/B}+\Delta)- d F +\Gamma_2 
\end{align*}
and, hence,
\[
K_{Y_b}+ (\Gamma_1)_b + g_b^*H_b- dF_b
\sim_{\Q}
(d/t)\left(-g_b^*(K_{X_b}+\Delta_b)-tF_b +(t/d)(\Gamma_2)_b \right).\]
For any sufficiently divisible $m\in \Z_{>0}$, every effective Cartier divisor $G\in |m(-g_b^*(K_{X_b}+\Delta_b)-tF_b+(t/d)(\Gamma_2)_b)|$ satisfies that $G+mtF_b \in |m(-g_b^*(K_{X_b}+\Delta_b)+(t/d)(\Gamma_2)_b)|$. 
Since $(\Gamma_2)_b$ is $g_b$-exceptional, we know that 
\[
|m(-g_b^*(K_{X_b}+\Delta_b)+(t/d)(\Gamma_2)_b)|=g_b^*|m(-K_{X_b}-\Delta_b)| + (mt/d)(\Gamma_2)_b.
\]
Hence $G+mtF_b\geq (mt/d)(\Gamma_2)_b$. 
Since $F_b \not\subset \Supp(\Gamma_2)$, we know that $G-(mt/d)(\Gamma_2)_b$ is an effective Cartier divisor in $|m(-g_b^*(K_{X_b}+\Delta_b)-tF_b)|$. Thus we have
\begin{align*}
    \vol\left(-g_b^*(K_{X_b}+\Delta_b)-tF_b\right)
    &= 
    \vol\left(-g_b^*(K_{X_b}+\Delta_b)-tF_b+(t/d)(\Gamma_2)_b\right)\\
&=
(t/d)^n\vol\left(K_{Y_b}+\Gamma_b +g_b^*H_b-dF_b\right).
\end{align*}
Hence, \eqref{e-volind} is independent of $b\in B$. Since this holds for each $t\in \Q_{>0}$,  $S_{X_b,\Delta_b}(F_b)$ and $T_{X_b,\Delta_b}(F_b)$ are also independent of $b\in B$. 
\end{proof}

For each closed point $b\in B$, we consider the infimum of $\frac{A}{S}$ and $\frac{A}{T}$ over the lc centers of $(X_b,\Delta_b+D_b)$. Specifically, we set
\begin{equation}\label{eq:a_b,d_b}
a_b:= \inf_{v}
   \frac{A_{X_b,\Delta_b}(v)}{T_{X_b,\Delta_b}(v)}
\quad 
\text{ and }
\quad
d_b:= \inf_{
v 
}
   \frac{A_{X_b,\Delta_b}(v)}{S_{X_b,\Delta_b}(v)}
,\end{equation}
where the infima run through $v\in \Val_{X_b}^*$ with  $ A_{X_b,\Delta_b+D_b}(v)=0$.
Using Proposition \ref{p-invarianceST}, we can understand how these values vary in families.

\begin{prop}\label{p-ind-ad}
If $(X,\Delta+D)\to B$ admits a fiberwise log
resolution, then $a_b$ and $d_b$ are independent  of the closed point $b\in B$.
Furthermore, each infimum is a minimum and achieved by a quasi-monomial valuation.
\end{prop}

\begin{proof}
Let $g:Y\to X$ be a fiberwise log resolution of $(X,\Delta+D)$
and define $\Gamma$ by the formula
$$K_{Y}+\Gamma = g^*(K_X+\Delta+D).$$
Since $g$ is a fiberwise log resolution,
$\Supp(\Gamma)$ is snc and each stratum of $\Supp(\Gamma)$  is
smooth with irreducible fibers over $B$. 
Write $\Gamma^{=1}$ for the sum of the prime divisors with coefficient 1 in $\Gamma$. 

For $b\in B$,
$K_{Y_b}+\Gamma_b = g^*(K_{X_b}+\Delta_b+D_b)$
and 
the valuations on $X_b$ that are lc places of $(X_b,\Delta_b+D_b)$
are precisely the valuations in $\QM(Y_b,\Gamma_b^{=1})$ by Lemma \ref{l-lcplacesQM}.
Note that $A$, $S$, and $T$ are continuous on 
$\QM(Y_b,\Gamma_b^{=1})$ (see Proposition \ref{prop:STcontinuousQM} for $S$ and $T$)
and homogeneous of degree 1. 
Therefore, $\frac{A}{T}$ and $\frac{A}{S}$ are continuous and  homogeneous of degree zero on
$\QM(Y_b,\Gamma_b^{=1})^*$.
This implies that the
functions achieve  minima on the latter set.

Next,
observe that
there is a natural isomorphism of simplicial cone complexes 
$$\QM(Y,\Gamma^{=1}) \simeq \QM(Y_b,\Gamma^{=1}_b),$$
since the stratum of $\Supp(\Gamma^{=1})$ 
are smooth over $B$ and have irreducible fibers. 
To finish the proof, we will show that the map
$$\QM(Y,\Gamma^{=1})^*\overset{\sim}{\to} \QM(Y_b,\Gamma_b^{=1})^*\overset{\frac{A}{S}}{\to} \R_{>0}$$
 is independent of $b\in B$ and the same holds for $\frac{A}{T}$. 
 
It is clear that  $\QM(Y,\Gamma^{=1})^*
\overset{\sim}{\to}
\QM(Y_b,\Gamma_b^{=1})^*\overset{A}{\to} \R_{>0}$ is independent of $b\in B$,
since 
$$(K_{Y}-g^*(K_X+D))\vert_{X_b} = K_{Y_b}-g_b^*(K_{X_b}+D_b).$$
Additionally, Proposition \ref{p-invarianceST} implies 
 $\QM(Y,\Gamma^{=1})^*
\overset{\sim}{\to}
\QM(Y_b,\Gamma_b^{=1})^*\overset{S}{\to} \R_{>0}$
is independent of $b\in B$ along the  rational points of $\QM(Y,\Gamma^{=1})^*$ and the same holds for $T$. Since $S$ and $T$ are continuous 
on  cones, the statement holds for all points in $\QM(\Gamma^{=1})$ and the proof is complete.
\end{proof}

\subsection{Constructibility of thresholds} 

We are now ready to prove that the stability threshold and global log canonical threshold are constructible in families. 

\begin{prop}\label{p-constrible} 
If $(X,\Delta)\to B$ is a $\Q$-Gorenstein family of log Fano pairs
over a normal variety $B$, then the functions
\[
B \ni b \mapsto \min\{ \alpha(X_{\overline{b}},\Delta_{\overline{b}}),1 \}
\quad
\text{ and }
\quad
B \ni b \mapsto \min\{\delta(X_{\overline{b}},\Delta_{\overline{b}}),1 \}
\]
are constructible.
\end{prop}

\begin{proof}
We only prove the statement for the stability threshold,
since the statement for the global log canonical threshold 
follows from the same argument, but with Corollary \ref{c-Ncompapproxd} replaced by Corollary \ref{c-Ncompapproxa}.

Fix a positive integer $r$ so that $r(K_X+\Delta)$ is a Cartier divisor.
Next, apply Corollary \ref{c-Ncompapproxd} to find a positive integer $N$ so that the following holds: if $b\in B$ is a closed point and $\d(X_b,\Delta_b)\leq 1$, then $$\d(X_b,\Delta_b)=\inf_{v} \frac{A_{X_b,\Delta_b}(v)}{S_{X_b,\Delta_b}(v)}$$
where the infimum runs through all divisorial valuations $v\in {\rm Val}_{X_b}^*$  that are lc places of an $N$-complement. 
Notice that for such a valuation $v$,
there exists $D_b\in {\frac{1}{rN} |-rN(K_{X_b}+\Delta_b)|}$ such that 
 $(X_b, \Delta_b+D_b)$ is lc and $A_{X_b, \Delta_b+D_b}(v)=0$.
 
To parameterize such boundaries, 
observe that  $f_* \cO_{X}(-rN(K_{X/B}+\Delta))$ commutes with base change,
since Kawamata-Viehweg vanishing implies ${H^1(X_b, \cO_{X_b}(-rN(K_{X_b}+\Delta_b)))}{=0}$  for all $b\in B$.
Set 
$$W:= \mathbb{P}(f_* \cO_{X}(-rN(K_{X/B}+\Delta))^*) \to B $$
and note that, for $b\in B$, 
the $k(b)$-points of $W_b$ are in bijection with divisors in $|-rN(K_{X_b}+\Delta_b)|$. 
Let $H$ be the universal divisor on $X\times_B W$ under this correspondence and set $D: = \tfrac{1}{rN}H$.
By the lower semicontinuity of the log canonical threshold, the locus
\[
Z:=\{w\in W \, \vert \, \lct(X_w,\Delta_w;D_w )=1 \}
\]
is locally closed in $W$. 
The scheme $Z$ together with the $\Q$-divisor $D_Z$ on $X_Z : = X\times_B Z$.
parameterizes boundaries of the desired form.

For a closed point $z\in Z$, set
%$$
%d_z:=
%\inf_v
%\frac{A_{X_z,\Delta_z}(v)}{S_{X_z,\Delta_z}(v)},$$ 
%where the infimum runs through 
%valutions $ v\in \Val_{X_z}^*$
 %such that $A_{X_z,\Delta_z+D_z}(v)=0$. %We assume $d_z=+\infty$ if there is no such $v$.
 $$
d_z:=
\inf \bigg\{ 
\frac{A_{X_z,\Delta_z}(v)}{S_{X_z,\Delta_z}(v)} 
\Big\vert 
v\in \Val_{X_z}^* \text{ and }
A_{X_z,\Delta_z+D_z}(v)=0 \bigg\}
,$$ 
By the above discussion, if $b\in B$ is closed, 
then $ \min \{1 ,\d(X_b,\Delta_b)\}$ equals the infimum of $  \{1 \} \cup \{ d_z \,\vert \, z \in Z_b \} $. 

Now, choose a locally closed decomposition $Z= \cup_{i=1}^r Z_i$ so that each $Z_i$
is smooth and there is an \'etale map $Z'_i \to Z_i$ such that $(X_{Z'_i},\Delta_{Z'_i}+D_{Z'_i})$ admits a fiberwise log resolution.
For a closed point $z\in Z_i$, $d_z$ is independent of $z\in Z_i$
(by Proposition \ref{p-ind-ad})
and we denote the value by $d^{(i)}$.
Hence,  
for  a closed point $b\in B$, $\min\{1,\d(X_b,\Delta_b)\}$
is the minimum of $\{1 \} \cup  \{ d^{(i)} \, \vert \, b\in \pi(Z_i)\}$. 
Therefore, we may write $B= \cup_j B_j$ as a finite union of constructible subsets
such that $B_j \ni b \mapsto \min \{1, \d(X_{\overline{b}},\Delta_{\overline{b}} )\}$ is constant on closed points.
Since the latter function is lower semicontinuous \cite{BL18}, it must be constant on all scheme theoretic points
and the proof is complete.
\end{proof}

\begin{proof}[Proof of Theorem \ref{t-constructible}]
Proposition \ref{p-constrible} implies that the functions are constructible. The main result of  \cite{BL18} implies they are lower semicontinuous. 
\end{proof}

\begin{rem}
To deduce Theorem \ref{t-constructible}, we do not need the full strength of \cite{BL18}. 
Indeed, we only need that the functions $\min\{1,\alpha\}$
and $\min \{1, \delta\}$ are weakly lower semicontinuous, which means that they do not increase under specialization.
\end{rem}

\begin{proof}[Proof of Corollary \ref{c-openness}]
Since a log Fano pair is K-semistable if and only if $\delta \geq 1$,
$B^\circ =\{ b\in B \, \vert\, \delta(X_{\overline{b}},\Delta_{\overline{b}})\geq 1\}$.
Theorem \ref{t-constructible} implies that the latter set is a Zariski open subset of $B$. 
\end{proof}

\begin{proof}[Proof of Theorem \ref{t-moduli}]
The proof of \cite[Cor. 1.4]{BX18}, but with the openness of uniform K-stability  replaced by the openness of K-semistability (Corollary \ref{c-openness}), implies $\mathfrak{X}^{\rm Kss}_{V,n}$
is an Artin stack of finite type over $k$.
With that step complete, we may apply \cite[Cor. 1.2]{ABHLX19} to see  $\mathfrak{X}^{\rm Kss}_{V,n}$ admits a good moduli space. 
By \cite[Rem. 2.2]{ABHLX19},  $k$-points of the good moduli space are in bijection with closed $k$-points of $\mathfrak{X}^{\rm Kss}_{V,n}$. 

It remains to show a $k$-point $[X] \in \mathfrak{X}^{\rm Kss}_{V,n}$ is closed if and only if $X$ is K-polystable. 
By \cite{LWX19} any K-semistable $\Q$-Fano variety degenerates to a uniquely determined K-polystable $\Q$-Fano variety via a special test configuration. 
Hence, if $[X]$ is closed, $X$ must be K-polystable.
Next, assume $X$ is K-polystable. 
Choose a closed $k$-point $[X_0]$ in the closure of $[X]$. Since  $[X_0]$ is closed, it is K-polystable.
Therefore, \cite{BX18} implies $X\simeq X_0$ and, hence, $[X]$ is closed. This completes the proof. %\footnote{\HB{added. while easy to deduce from known results, it doesn't seem to be explicitly explained in the literature.}}
\end{proof}

As a consequence of Corollary \ref{c-openness} and Theorem \ref{t-moduli}, we deduce the openness of K-stability and constructibility of K-polystability.

\begin{thm}\label{t-openkstable}
Let $(X,\Delta)\to B$ be a  $\Q$-Gorenstein 
family of log Fano pairs
over a normal base $B$. Then the set 
\[
B^{\rm Ks}: = \{ b \in B \, \vert \, (X_{\overline{b}},\Delta_{\overline{b}})
\text{ is  K-stable} \}
\]
is a Zariski open subset of $B$. Moreover, the set
\[
B^{\rm Kps}: = \{ b \in B \, \vert \, (X_{\overline{b}},\Delta_{\overline{b}})
\text{ is  K-polystable} \}
\]
is a constructible subset of $B$.
\end{thm}

\begin{proof}
For simplicity, we assume $\Delta=0$ as the proof for the general case follows similarly by replacing the K-moduli space of $\bQ$-Fano varities by log Fano pairs (see e.g. \cite[Theorem 2.21]{XZ19}). We may also assume that $B$ is irreducible. 
Let $n:=\dim X_{b}$ and $V:= (-K_{X_b})^n$ for a closed point $b\in B$. By Corollary \ref{c-openness}, there is an open subset $B^\circ$ of $B$ parametrizing K-semistable fibers. Denote by $X^\circ:=X\times_B B^\circ$. We take the Artin stack $\fX_{n,V}^{\rm Kss}$ and its separated good moduli space $X_{n,V}^{\rm Kps}$ from Theorem \ref{t-moduli}. Since the Artin stack $\fX_{n,V}^{\rm Kss}$ represents the moduli functor of K-semistable $\bQ$-Fano varieties, there is a morphism $\phi: B^\circ\to \fX_{n,V}^{\rm Kss}$ whose pull-back of the universal family over $\fX_{n,V}^{\rm Kss}$ gives $X^\circ\to B^\circ$.

As shown in the proof of Theorem \ref{t-moduli},  K-polystable $\bQ$-Fano varieties of dimension $n$ and volume $V$ corresponds to closed points of the stack $\fX_{n,V}^{\rm Kss}$. %\footnote{\HB{Does it make sense to also say [BX19]? (i.e. LWX18 shows that K-ps Fanos do not degenerate to K-ss Fanos along test configs... though using the next sentence (that closed points of the stack have reductive stabilizers), one doesn't need [BX19]) }}
Let $x$ be such a closed point with stabilizer group $G_x$ which is reductive by \cite{ABHLX19}. Since $\fX_{n,V}^{\rm Kss}$ is a global quotient stack, by the Luna \'etale slice theorem  \cite[Remark 2.11]{ABHLX19}, there exist an affine scheme $\Spec(A)$ with an action of $G_x$, a $G_x$-fixed closed point $w\in \Spec(A)$, and  a Cartesian diagram \begin{center}
    \begin{tikzcd}
    {[\Spec(A)/G_x]} \arrow[d]\arrow[r, "f"] & \fX_{n,V}^{\rm Kss}\arrow[d, "\pi"]\\
    \Spec(A)\sslash G_x \arrow[r] & X_{n,V}^{\rm Kps}
    \end{tikzcd}
\end{center}
such that $f([w])=[x]$, $\Spec(A)\sslash G_x$ is an \'etale neighborhood of $\pi(x)$,  $f$ sends closed points to closed points, and $f$ induces an isomorphism of stabilizer groups at closed points.
%\footnote{\HB{Added that this is true after shrinking. While I assume this is likely actually true to begin with since the bottom map is also \'etale, this isn't quite explained in [ABHLX19]. }}
In particular, we know that a geometric point $y\in \Spec(A)$ is GIT stable (resp. GIT polystable) if and only if $f([y])$ represents a K-stable (resp. K-polystable) $\bQ$-Fano variety. Since GIT stable locus  is open (see \cite[Chapter 1.4]{GIT} or \cite[Proposition 5.15]{Muk03}), and openness is an \'etale-local property, we know that $B^{\rm Ks}$ is an open subset of $B^\circ$. Similarly, it is a well-known fact that the GIT polystable locus is constructible. %(see e.g. \cite[Proposition 2.14]{GIT} and \cite[Section 9.1]{LWX19}), and 
Here we give a brief proof of this fact. Denote by $Z:= \Spec(A)$ and $\Phi: Z\to Z\sslash G_x$ the good quotient morphism. Let $Z_r:= \{z\in Z\mid \dim \mathrm{Stab}(z)\geq r\}$. Then it is clear that $Z_r$ is a closed subset of $Z$. Let $Z^{\rm ps}$ be the GIT polystable locus of $Z$. Then by similar argument to \cite[Proposition 5.13]{Muk03}, we know that $Z^{\rm ps}= \cup_{r=0}^{\dim G_x} (Z_r\setminus \Phi^{-1}(\Phi(Z_{r+1})))$. Hence $Z^{\rm ps}$ is constructible in $Z$. Since constructibility is an \'etale-local property, we have that $B^{\rm Kps}$ is a constructible subset of $B^\circ$. The proof is finished.
\end{proof}

\subsection{Existence of valuations computing the stability threshold}
A consequence of the proof of Theorem \ref{t-constructible} is the following result.
In Theorem \ref{t-minimizer=lcplaceN}, we will obtain properties of valuations computing $\delta\leq 1$.

\begin{thm}\label{t-quasimonomialless1}
If $(X,\Delta)$ is a log Fano pair with $\delta(X,\Delta)\le1$, then there exists a quasi-monomial valuation $v\in \Val_X^*$ computing $\d(X,\Delta)$.
\end{thm}

The existence of valuations computing the stability threshold was previously proven  in \cite{BJ17} 
using the generic limit construction. The proof below is entirely different.

\begin{proof}
Fix a positive integer $r$ so that $r(K_X+\Delta)$ is a Cartier divisor
and  apply Corollary \ref{c-Ncompapproxd}
%\footnote{\HB{we need to add a statement in the previous section that improves the result to the $\delta=1$ case}}
to find a positive integer $N$ so that 
 $\d(X,\Delta)=\inf_{v} \frac{A_{X,\Delta}(v)}{S(v)}$,
where the infimum runs through all valuations $v\in {\rm DivVal}_{X}$  that are lc places of an $N$-complement. For such a valuation $v$,
there exists $D \in {\frac{1}{rN} |-rN(K_{X}+\Delta)|}$ such that 
 $(X, \Delta+D)$ is lc and $A_{X, \Delta+D}(v)=0$. We proceed to parameterize such $\Q$-divisors.
 
 Set $W: =\mathbb{P}\left( H^0(X, \cO_{X}(-rN(K_X+\Delta)))^*\right)$. Write $H$ for the universal divisor on $X\times W$
 parameterizing divisors in $|-rN(K_X+\Delta)|$ and set $D: = \frac{1}{rN}H$. 
By the lower semicontinuity of the log canonical threshold, the locus
\[
Z:=\{w\in W \, \vert \, \lct(X_w,\Delta_w;D_w )=1 \}
\]
is locally closed in $W$. 
Hence,  the $\Q$-divisor $D_Z$
on $X\times Z$
 parameterizes boundaries of the desired form.

For each closed point $z\in Z$, set
$d_z:= \inf_{v} \frac{A_{X,\Delta}(v)}{S(v)}$, where the infimum runs through all 
$v\in \Val_{X_z}$ such that 
$A_{X,\Delta+D_z}(v)=0$. 
This  infimum is a minimum by Proposition \ref{p-ind-ad} and is achieved by a quasi-monomial valuation $v_{z}^*$. 
By the above discussion, 
 $ \d(X,\Delta)$ equals $\inf\{ d_z \, \vert \, z\in Z\}$.
 
Now, choose a locally closed decomposition $Z= \cup_{i=1}^r Z_i$ so that each $Z_i$
is smooth and there is an \'etale map $Z'_i \to Z_i$ such that $(X_{Z'_i},\Delta_{Z'_i}+D_{Z'_i})$ admits a fiberwise log resolution.
For closed point $z\in Z_i$, $d_z$ is independent of $z\in Z_i$
by Proposition \ref{p-ind-ad}.
Therefore, $d_z$ takes finitely many values
and we can find  $z_0 \in Z$ such that  
$\d(X,\Delta) = d_{z_0} $ and is computed by $v_{z_0}^*$.
\end{proof}

\begin{rem}
The proof of the existence of valuations computing the stability threshold in \cite{BJ17} requires 
that the base field be uncountable. The assumption is not needed in the above proof.
\end{rem}

\appendix
\section{K-stability and complements}\label{s-weaklyspecial}
In the appendix, we will combine the cone construction and results from \cite{Xu19} to further use complements to study the K-stability  of a log Fano pair $(X,\Delta)$. 
During our investigation, we will relate degenerations of the log Fano pair $(X,\Delta)$ to valuations centered on the vertex of the cone. This idea was first introduced in \cite{Li17} and then extended in \cite{LX16, LWX18} in the study of relations between K-stability and normalized volume. It is a powerful technique and works particularly well for studying the valuations $v$ which computes $\delta(X,\Delta)=1$ (see e.g. \cite{BX18}). Following the spirit of \cite{Li17, LX16, LWX18}, we obtain results on the log Fano pair $(X,\Delta)$ by applying the local results from \cite{Xu19}  to cone singularities.  We note that \cite{Xu19} and the current paper have a similar strategy via the local-to-global principle. Both papers use the existence of bounded complements proved in \cite{Bir16a}.

\subsection{Test configuration and lc places}

In \cite{Tia97,Don02}, the K-(semi,poly)stability of a log Fano pair is defined by looking at the sign of {\it the generalized Futaki invariant} of every {\it test configuration}. For definitions and more background, see e.g. \cite{LX14}, \cite{BHJ17}. 

\begin{defn}\label{d-weaklyspecial}
Let $(X,\Delta)$ be a log Fano pair. 
A test configuration (resp. semiample test configuration) $(\cX,\Delta_{\tc};\cL)$ of $(X,\Delta)$ is said to be \emph{weakly special}
if $(\cX,\Delta_{\tc}+\cX_0)$ is log canonical and $\cL\sim_{\bQ} -K_{\cX}-\Delta_{\tc}$ is ample (resp. semiample) over $\bA^1$.
A finite set of (possibly trivial) $\Z$-valued divisorial valuations
%\footnote{I am assuming that divisorial valuations are of the form $c\cdot\ord_E$ with $c\in\Z_{\geq 0}$ -- YL} 
$\{v_1,\cdots, v_d\}\subset\Val_X$ is called a \emph{weakly special collection}
if there exists a weakly special semiample test configuration $(\cX,\Delta_{\tc};\cL)$ of $(X,\Delta)$
such that $v_i=  v_{\cX_0^{(i)}}$ (see \cite[Def. 4.4]{BHJ17})
%for some $d_i\in \bQ_{>0}$\footnote{\CX{I add a coefficient}}  
where $\{\cX_0^{(i)}\}_{i=1}^d$ are all
irreducible components of $\cX_0$. 
A \emph{prime} divisor $E$ over $(X,\Delta)$ is said to be \emph{weakly special} 
if there exists a weakly special test configuration 
$(\cX,\Delta_{\tc};\cL)$ of $(X,\Delta)$ with $\cX_0$ irreducible 
such that $v_{\cX_0}=c\cdot\ord_E$ for some $c\in\Z_{>0}$.
\end{defn}

\begin{thm}\label{thm:weaklyspecial}
Let $n$ be a positive integer and $I\subset \bQ\cap[0,1]$ a finite set. Then there exists a positive integer $N=N(n,I)$ satisfying the following:

If $(X,\Delta)$ is an $n$-dimensional log Fano pair such that coefficients of $\Delta$ belong to $I$, then a finite set of $\Z$-valued divisorial valuations
$\{v_1,\cdots, v_d\}\subset\Val_X$ is a weakly special collection
if and only if there exists an $N$-complement $\Delta^+$ of $(X,\Delta)$ such that each $v_i$ is an lc place of 
$(X,\Delta^+)$.
\end{thm}

Note that a special case of Theorem \ref{thm:weaklyspecial} on weakly special divisors with small $\beta$-invariants is proved in \cite[Theorem 3.10]{ZZ19} independently. %\footnote{\YL{I mention [ZZ19] since they use the global approach}}. 

Our argument is a refinement of \cite[Proof of Lemma 3.4]{LWX18}. In particular, we track the integral structure of test configurations and complements. 

\medskip

We first show that any weakly special collection of $\Z$-valued divisorial
valuations $\{v_1,\cdots,v_d\}$ are lc places of a common $\bQ$-complement.
We will use the following cone construction: Fix a positive integer $m$ such that $L:=-m(K_X+\Delta)$ is Cartier. Let $Z:=C(X,L)$ be the affine cone over $X$ with polarization $L$. Denote by  $o\in Z$ the cone vertex. Let $\Gamma$ be the Zariski closure of the pull-back of $\Delta$ under the projection $Z\setminus\{o\}\to X$. 
Denote by $w_0:=\ord_{X_0}$ the canonical valuation in $\Val_Z$ where $X_0$ is the exceptional divisor of blowing up the cone vertex $o$ in $Z$.
Assume $v_i=c_i\cdot\ord_{E_i}$ with $c_i\in\Z_{\geq 0}$. 
Denote by $k_0:=m\cdot\max_{1\leq i\leq d}\{A_{X,\Delta}(v_i)\}$.
For each divisor $E_i$ over $X$ and each integer $k> k_0$, we consider the 
divisorial valuation $w_{i,k}$ on $Z$ as a quasi-monomial combination of $\ord_{X_0}$ and $\ord_{E_{i,\infty}}$ with weights $(1-\frac{mA_{X,\Delta}(v_i)}{k},\frac{c_i}{k})$ where $E_{i,\infty}$ is the pull-back of $E_i$ under the projection $Z\setminus\{o\}\to X$.
Since $m A_{X,\Delta}(E_i)$ is a positive integer, we know that the value group of $w_{i,k}$ is generated by $1$ and $\frac{c_i}{k}$.
Thus for any $k>k_0$ there is a prime divisor $E_{i,k}$ over $Z$ centered at $o$ such that $\frac{k}{\gcd(k,c_i)}\cdot w_{i, k}=\ord_{E_{i,k}}$.  

\begin{prop}\label{p-extracting}
Suppose $\{v_1,\cdots,v_d\}$ is a weakly special collection of $\Z$-valued divisorial valuations over $(X,\Delta)$. Then for any $k$ sufficiently large, there exists a proper birational morphism $\mu_k:W_k\to Z$ from a normal variety $W_k$ such that 
\begin{enumerate}
    \item $\mu_k$ is an isomorphism over $Z\setminus\{o\}$ and $\mu_k^{-1}(o)=\cup_{i=1}^d E_{i,k}$;
    \item $(W_k,(\mu_k)_*^{-1}\Gamma + \sum_{i=1}^d E_{i,k})$ is log canonical;
    \item $-(K_{W_k}+(\mu_k)_*^{-1}\Gamma + \sum_{i=1}^d E_{i,k})$ is semiample over $Z$.
\end{enumerate}
\end{prop}

\begin{proof}
Let $(\cX,\Delta_{\tc})$ be the weakly special semiample test configuration corresponding to $\{v_i\}_{i=1}^d$. 
For $L=-m(K_X+\Delta)$, we denote by 
\[
R=\oplus_{j=0}^\infty R_j:=\oplus_{j=0}^\infty H^0(X,jL).
\]
It is clear that $Z=\Spec(R)$. Let us take an ample model $\rho:\cX\to \cX'$ of $-K_{\cX}-\Delta_{\tc}$ over $\bA^1$. Denote by $\Delta_{\tc}':=\rho_*\Delta_{\tc}$, then it is clear that $(\cX',\Delta_{\tc}')$ is a weakly special test configuration of $(X,\Delta)$. After reindexing we can assume that $\rho$ precisely contracts $\cX_0^{(i)}$ for $d'<i\leq d$ where $d'$ is the number of irreducible components of $\cX_0'$. 

Before constructing $\mu_k: W_k\to Z$, we first construct $\mu_k':W_k'\to Z$ such that (1),(2), and (3) hold as well after replacing $(\mu_k, W_k, d)$ and ``semiample'' by $(\mu_k', W_k', d')$ and ``ample'', respectively. We denote these new statements by (1'), (2'), and (3').

For part (1'), consider the following 
 $\Z$-filtration of $R$
\begin{equation}\label{eq:lc-blowup-filt}
\cF^{p}R_j:= \{s\in R_j\mid v_i(s)\geq p+mA_{X,\Delta}(v_i)j\quad\textrm{ for any } 1\leq i\leq d'\}.
\end{equation}
By \cite[Propositions 2.15, 4.11, and Lemma 5.17]{BHJ17}, we know that the filtration $\cF^{\bullet} R$ is finitely generated,
\begin{equation}\label{eq:tc-filt}
\cX'\cong\Proj\bigoplus_{j=0}^\infty \bigoplus_{p=-\infty}^{\infty} t^{-p}\cF^p R_j,
\end{equation}
and $-m(K_{\cX'/\bA^1}+\Delta_{\tc}')$ corresponds to $\cO(1)$ under the grading of $j$. Consider the ideal sequence
$
I_p:=\oplus_{j=0}^\infty \cF^{p-jk} R_j\subset R
$
for $p\in \Z$. By \eqref{eq:lc-blowup-filt} we see that $\cF^{p-jk}R_j=R_j$, whenever $j\geq p/(k-mA_{X,\Delta}(v_i))$ for all $1\leq i\leq d'$. So $I_p$ is cosupported at $o$ if $p>0$, and $I_p=R$ if $p\leq 0$. 
Since $\cF^{\bullet}R$ is finitely generated and multiplicative, we see that $\oplus_{p=0}^\infty I_p$ is also a finitely generated $R$-algebra. Let $W_k':=\Proj_R \oplus_{p=0}^\infty I_p$ with $\mu_k':W_k'\to Z$ the projection morphism.
%The normality of $W_k'$ follows from the normality of $\cX'$ and the fact that their projective coordinate rings are isomorphic up to a grading shift. 

Next we show that $I_p=\cap_{i=1}^{d'} \fa_{p/k}(w_{i,k})$. Since $w_{i,k}$ is $\bG_m$-invariant, its valuation ideals are graded. Hence it suffices to verify the above equality for all homogeneous elements. Let $s\in R_j$ be a homogeneous element. From the definition of $w_{i,k}$ we know
\[
w_{i,k}(s)=\left(1-\frac{mA_{X,\Delta}(v_i)}{k}\right)j + \frac{c_i}{k}\ord_{E_i}(s)=\frac{1}{k}\big(jk- mA_{X,\Delta}(v_i)j + v_i(s)\big).
\]
Thus $w_{i,k}(s)\geq p/k$ if and only if $v_i(s)\geq p-jk+mA_{X,\Delta}(v_i)j$, which implies $I_p=\cap_{i=1}^{d'} \fa_{p/k}(w_{i,k})$. Since for each $p\geq 0$ the ideal $I_p$ is integrally closed as a finite intersection of valuation ideals, we know that $W_k'$ is normal. Besides, by \cite[Lemma 5.17]{BHJ17} if any $i$ is dropped from the intersection on the right-hand side of \eqref{eq:lc-blowup-filt} we would not get $\cF^p R_j$. Hence \cite[Theorem 1.10]{BHJ17} implies that for $p$ sufficiently divisible, the set of Rees valuations of $I_p$ is given by $\{\frac{k}{p}w_{i,k}\}_{i=1}^{d'}$. Thus $\mu_k':W_k'\to Z$ precisely extracts $\cup_{i=1}^{d'} E_{i,k}'$ where $E_{i,k}'$ is the birational transform of $E_{i,k}$ and we confirm part (1').

For part (3'), we know that 
\begin{align*}
A_{Z,\Gamma}(w_{i,k})&= \left(1-\frac{mA_{X,\Delta}(v_i)}{k}\right)A_{Z,\Gamma}(X_0)+\frac{c_i}{k}A_{Z,\Gamma}(E_{i,\infty})\\
& = \left(1-\frac{mA_{X,\Delta}(v_i)}{k}\right)\frac{1}{m}+\frac{1}{k}A_{X,\Delta}(v_i)=\frac{1}{m}.
\end{align*}
Hence $A_{Z,\Gamma}(E_{i,k})=\frac{k}{m\gcd(k,c_i)}$. Straight computation shows
\begin{equation}\label{eq:W_k'}
K_{W_k'}+(\mu_k')_*^{-1}\Gamma + \sum_{i=1}^{d'} E_{i,k}'=\mu_k'^*(K_Z+\Gamma)+\sum_{i=1}^{d'} A_{Z,\Gamma}(E_{i,k}) E_{i,k}'\sim_{\mu_k',\Q}\sum_{i=1}^{d'} \frac{kE_{i,k}'}{m\gcd(k,c_i)}.
\end{equation}
From the above discussion on Rees valuations, we know that $\cO_{W_k'}(-1)=\sum_{i=1}^{d'}\frac{E_{i,k}'}{\gcd(k,c_i)}$ is anti-ample over $Z$. Thus  $-(K_{W_k'}+(\mu_k')_*^{-1}\Gamma + \sum_{i=1}^{d'} E_{i,k}')\sim_{\mu_k',\bQ} \frac{k}{m}\cO_{W_k'}(1)$ is ample over $Z$ which confirms part (3'). 

For part (2'), notice that the ideal sequence $I_\bullet$ induces a $\bG_m$-equivariant degeneration $\cZ\to \bA^1$ of $o\in Z$, where $\cZ:=\Spec \oplus_{p\in \Z} t^{-p} I_p$. Then we have the central fiber $Z_0=\Spec( \gr_{I_\bullet} R)$ where 
\begin{equation}\label{eq:cone-tc-1}
\gr_{I_\bullet} R:= \bigoplus_{p=0}^\infty I_p/I_{p+1}
=\bigoplus_{p=-\infty}^\infty \bigoplus_{j=0}^\infty \cF^{p-jk} R_j/\cF^{p-jk+1}R_j.
\end{equation}
Here we are using the fact that $\cF^{p-jk}R_j=R_j$ whenever $p\leq 0$. From \eqref{eq:tc-filt} we know that 
\begin{equation}\label{eq:cone-tc-2}
\cX_0'\cong \Proj\bigoplus_{j=0}^\infty \bigoplus_{p=-\infty}^{\infty} \cF^p R_j/\cF^{p+1}R_j
\end{equation}
and $-m(K_{\cX_0'}+\Delta'_{\tc,0})$ corresponds to $\cO(1)$ under the grading of $j$. It is clear that $\gr_{I_\bullet} R$ is isomorphic to $\oplus_{j=0}^\infty \oplus_{p=-\infty}^{\infty} \cF^p R_j/\cF^{p+1}R_j$ up to a grading shift. Let $\Gamma_{\cZ}$ be the effective $\bQ$-divisor on $\cZ$ as the Zariski closure of $\Gamma\times (\bA^1\setminus\{0\})$. Denote by $\Gamma_0:=\Gamma_{\cZ}|_{Z_0}$ the degeneration of $\Gamma$ to $Z_0$. Then $(Z_0,\Gamma_0)$ is semi-log canonical (slc) since it is isomorphic to the affine cone over the slc pair $(\cX_0',\Delta_{\tc,0}')$ with the polarization $-m(K_{\cX_0'}+\Delta_{\tc,0}')$.
%\footnote{\HB{Do we need to replace $L$ for $Z_0$ to be the cone (rather than some orbifold cone). A  similar situation occurs in the proof of Proposition A.5. I realize this is minor. }\CX{Since the referee doesn't complain here, I think we can just leave as it is.}}. 
Thus we know that $(\cZ, \Gamma_{\cZ}, \xi;\eta)$ is a weakly special test configuration of $(Z,\Gamma, \xi)$ in the sense of \cite[Definition 2.14]{LWX18} where $\xi$ (resp. $\eta$) is the vector field on $Z$ (resp. $\cZ$) induced by the grading of $j$ (resp. of $p$). We will follow the idea of \cite[Proof of Lemma 2.21(2)]{LWX18} to show log canonicity of $(W_k',(\mu_k')_*^{-1}\Gamma +\sum_{i=1}^{d'} E_{i,k}')$.

%\textcolor{red}{Thus \cite[Lemma 2.21(2)]{LWX18} implies that $(W_k',(\mu_k')_*^{-1}\Gamma +\sum_{i=1}^{d'} E_{i,k}')$ is log canonical. This proves part (2').}

%It has been shown in \cite{Li17} that $(\sum_{i=1}^{d'} E_{i,k}', G_k')$ is the orbifold quotient of $(Z_0,\Gamma_0)$ by the $1$-PS corresponding to the grading of $p$, where $G_k'$ is the different of $(W_k',(\mu_k')_*^{-1}\Gamma +\sum_{i=1}^{d'} E_{i,k}')$ along $\sum_{i=1}^{d'} E_{i,k}'$. Hence $(\sum_{i=1}^{d'} E_{i,k}', G_k')$ is also slc which confirms part (2') by inversion of adjunction (see \cite[Theorem 4.9]{Kol13}).

%\textcolor{blue}{YL: The argument above for (2') needs some extra argument as pointed out by the referee. Below I'll give a detailed account.}

Denote by $E':=\sum_{i=1}^{d'} \frac{E_{i,k}'}{\gcd(k,c_i)}$ and $E'_{\red}:=\sum_{i=1}^{d'} E_{i,k}'$. From the proof of part (3') we know that $E'=\cO_{W_k'}(-1)$ is anti-ample over $Z$. Let $l$ be a sufficiently divisible positive integer such that $lE'$ is Cartier on $W_k'$. The test configuration $(\cZ,\Gamma_{\cZ},\xi;\eta)$ has the natural $\bG_m$-action generated by $\eta$. Consider the $\bmu_{l}$-action on $(\cZ,\Gamma_{\cZ})$ where $\bmu_l<\bG_m$ is the multiplicative group of $l$-th roots of unity. Let $(\cZ',\Gamma_{\cZ'}):=(\cZ,\Gamma_{\cZ})/\bmu_l$. By construction, we have that $\cZ':=\Spec\oplus_{p\in\Z} t^{-p} I_{lp}\to \bA^1_t$, such that the quotient map $\sigma:\cZ\to \cZ'$ is a lifting of the map $\bA_t^1\to \bA_t^1$, $t\mapsto t^l$. Clearly $\sigma$ is \'etale away from the central fibers. Since $Z_0=\Spec\oplus_{p\in \Z_{\geq 0}} I_p/I_{p+1}$, we know that $Z_0/\bmu_l=\Spec \oplus_{p\in \Z_{\geq 0}} I_{lp}/I_{lp+1}$, and $\Supp (Z_0/\bmu_l)=\Supp(Z_0')$.

Next, we show that $Z_0/\bmu_l \cong C_a(E'_{\red}, \cO_{E'_{\red}}(-lE'|_{E'_{\red}}))$ where $C_a(X,L)$ represents the affine cone over $X$ with polarization $L$ (see \cite[Section 3.1]{Kol13}). Indeed, from the equality  $I_p=\cap_{i=1}^{d'} \fa_{p/k}(w_{i,k})$ we see that $I_p=(\mu_k')_*\cO_{W_k'}(\lfloor -pE'\rfloor)$. Since $lE'$ is Cartier and $\lceil E'\rceil= E'_{\red}$, we know that $\lfloor -(lp+1) E'\rfloor= -lpE' - E'_{\red}$. Then we have a short exact sequence
\begin{equation}\label{eq:lc-blowup-exact-seq}
0\to \cO_{W_k'}(-lpE'-E'_{\red})\to \cO_{W_k'}(-lpE') \to \cO_{E'_{\red}} (-lpE'|_{E_{\red}'})\to 0.
\end{equation}
Since $l$ is sufficiently divisible and $-E'$ is ample over $Z$, we have $R^1 (\mu_k')_* \cO_{W_k'}(-lpE'-E'_{\red})=0$ for $p\geq 1$ by Serre vanishing. Thus taking $(\mu_k')_*$ of \eqref{eq:lc-blowup-exact-seq} yields a short exact sequence
\[
0\to I_{lp+1}\to I_{lp}\to H^0(E_{\red}', \cO_{E'_{\red}} (-lpE'|_{E_{\red}'}))\to 0,
\]
i.e. $I_{lp}/I_{lp+1}\cong H^0(E_{\red}', \cO_{E'_{\red}} (-lpE'|_{E_{\red}'}))$ when $p\geq 1$. If $p=0$, then the above arguments give an injection $I_0/I_1\hookrightarrow H^0(E_{\red}',\cO_{E_{\red}'})$ which implies that they are isomorphic   as $h^0(E_{\red}',\cO_{E_{\red}'})=1$ by reducedness of $E_{\red}'$. 

Since $(\cZ, \Gamma_{\cZ},\xi;\eta)$ is weakly special, we know that $(\cZ, \Gamma_{\cZ}+Z_0)$ is log canonical. In particular, we know that $Z_0$ is reduced and so is $Z_0/\bmu_l=(Z_0')_{\red}$.  Since the quotient map $\sigma: \cZ\to \cZ'$ is \'etale away from central fibers, we have that $K_{\cZ}+\Gamma_{\cZ}+ Z_0 = \sigma^*(K_{\cZ'} +\Gamma_{\cZ'} + (Z_0')_{\red})$.  Therefore, the quotient $(\cZ', \Gamma_{\cZ'} + (Z_0')_{\red})$ is also log canonical by \cite[Proposition 5.20]{KM98}. By adjunction we know that $(Z_0,\Gamma_0)/\bmu_l$ is slc. This implies that the base $(E_{\red}', \Gamma_{E_{\red}'})$ is slc where $K_{E_{\red}'}+\Gamma_{E_{\red}'}=(K_{W_k'}+(\mu_k')_*^{-1}\Gamma+E'_{\red})|_{E_{\red}'}$. By inversion of adjunction, the pair $(W_k', (\mu_k')_*^{-1}\Gamma+E'_{\red})$ is log canonical. This proves part (2').

So far we have proven (1'), (2'), and (3') for $\mu_k':W_k'\to Z$. In order to construct $\mu_k:W_k\to Z$, we will show that $E_{i',k}$ is an lc place of $(W_k',(\mu_k')_*^{-1}\Gamma + \sum_{i=1}^{d'} E_{i,k}')$ for any $d'<i'\leq d$. By \eqref{eq:W_k'}, we know
\begin{align*}
    A_{W_k',(\mu_k')_*^{-1}\Gamma + \sum_{i=1}^{d'} E_{i,k}'}(w_{i',k})& =A_{Z,\Gamma}(w_{i',k})-w_{i',k}\big(\sum_{i=1}^{d'}\frac{kE_{i,k}'}{m\gcd(k,c_i)}\big)\\
    & = \frac{1}{m}\left(1-kw_{i',k}(\cO_{W_k'}(-1))\right).\label{eq:weaklysp-lc-blowup}\numberthis
\end{align*}
Indeed, since $(\cX,\Delta_{\tc};\cL)$ is the pull-back test configuration of $(\cX',\Delta_{\tc}';\cL')$, by \cite[Lemma 2.13]{BHJ17} they define the same filtration, i.e.
\[
\cF^{p}R_j= \{s\in R_j\mid v_i(s)\geq p+mA_{X,\Delta}(v_i)j\quad\textrm{ for any } 1\leq i\leq d\}.
\]
Similar to the arguments above, we have $I_p=\cap_{i=1}^d\fa_{p/k}(w_{i,k})$. Hence $w_{i,k}(\cO_{W_k'}(-1))= \frac{1}{k}$ for any $1\leq i\leq d$. This together with \eqref{eq:weaklysp-lc-blowup} implies $w_{i',k}$ is an lc place of $(W_k',(\mu_k')_*^{-1}\Gamma + \sum_{i=1}^{d'} E_{i,k}')$ for any $d'<i'\leq d$. It is clear that all non-klt centers of $(W_k',(\mu_k')_*^{-1}\Gamma + \sum_{i=1}^{d'} E_{i,k}')$ are contained in $\cup_{i=1}^{d'} E_{i,k}'$, thus $W_k'$ is of Fano type over $Z$. 
Then \cite{BCHM10} implies that there exists a projective birational morphism $\rho_k:W_k\to W_k'$ from a normal variety $W_k$ such that $\Exc(\rho_k)=\cup_{d'<i'\leq d} E_{i',k}$. Moreover, we know that $K_{W_k}+(\mu_k)_*^{-1}\Gamma + \sum_{i=1}^d E_{i,k}$ is the log pull-back of $K_{W_k'}+(\mu_k')_*^{-1}\Gamma + \sum_{i=1}^{d'} E_{i,k}'$ since $\rho_k$ only extracts lc places of the latter. By taking $\mu_k:=\mu_k'\circ\rho_k$, it is easy to see that (1), (2), and (3) are all satisfied. Thus the proof is finished.
\end{proof}

\begin{prop}\label{p-local}
There exists a positive integer $N_1=N_1(n,I)$ such that the following holds: for any weakly special collection of $\Z$-valued divisorial valuations $\{v_1,\cdots,v_d\}$ over $(X,\Delta)$ where $\dim(X)=n$ and coefficients of $\Delta$ belongs to $I$, and any $k\gg 1$, there exists a local $N_1$-complement $\Gamma_k^+$ of $o\in (Z,\Gamma)$ such that $E_{i,k}$ is an lc place of $(Z,\Gamma_k^+)$ for any $1\leq i\leq d$.
\end{prop}

\begin{proof}
By applying the boundedness of relative complements \cite[Theorem 1.8]{Bir16a} to the morphism $\mu_k:W_k\to Z$ constructed in Proposition \ref{p-extracting}, there exists an $N_1$-complement $\Theta_k$ of $(W_k, (\mu_k)_*^{-1} \Gamma +\sum_{i=1}^d E_{i,k})$ over $o\in Z$ where $N_1$ only depends on the dimension $n$ and the coefficient set $I$. Then $\Gamma_k^+: = (\mu_k)_*(\Theta_k)$ is a local $N_1$-complement of $o\in (Z,\Gamma)$ such that $E_{i,k}$ is an lc place for any $1\leq i\leq d$. 
%\footnote{
%\HB{Maybe, we could add slightly more details (i.e. ``applying Bir19 to get a  local complement $\Theta_k$ of $(W_k, (\mu_k)_*^{-1} \Gamma +\sum_{i=1}^d E_{i,k})$. Then, $\Gamma_k^+: = (\mu_k)_*(\Theta_k)$, is a local complement... ). Also  we haven't defined the term local  complement in the preliminaries. We could add it or just write ``see Bir19'' (not sure if this matters). 
%}}
\end{proof}

\begin{proof}[Proof of Theorem \ref{thm:weaklyspecial}]
Let $\Gamma_k^+$ be the $N_1$-complement as in Proposition \ref{p-local}. Then we know that 
\[
A_{Z,\Gamma_k^+}(w_{i,k})=\frac{\gcd(k,c_i)}{k}A_{Z,\Gamma_k^+}(\ord_{E_{i,k}})=0.
\]
Let $r$ be the Gorenstein index of $o\in (Z,\Gamma)$. Then 
\[
rN_1(\Gamma_k^+-\Gamma)\sim rN_1(K_Z+\Gamma_k^+)-rN_1(K_Z+\Gamma)\sim 0.
\]
Thus we have $\Gamma_k^+=\Gamma+\frac{1}{rN_1}\mathrm{div}(f_k)$ where $f_k\in\cO_{o,Z}$. It is then clear that 
\[
A_{Z,\Gamma}(w_{i,k})=\frac{w_{i,k}(f_k)}{rN_1}.
\]
By definition we know that $w_{i,k}\geq (1-\frac{mA_{X,\Delta}(v_i)}{k})w_0$. On the other hand, for any $f\in R_j$ it is clear that 
\[
w_{i,k}(f)=\left(1-\frac{mA_{X,\Delta}(v_i)}{k}\right)j+\frac{c_i}{k}v_i(f)
\leq j+\frac{c_i m T(v_i)}{k} j= \left(1+\frac{c_i m T(v_i)}{k}\right) w_0(f).
\]
Hence there exists a sequence of positive numbers $\epsilon_k\to 0$ as $k\to\infty$ such that 
$(1-\epsilon_k) w_0\leq w_{i,k}\leq (1+\epsilon_k)w_0$ for any $1\leq i\leq d$. This implies
\[
A_{Z,\Gamma}(w_0)=\lim_{k\to\infty} A_{Z,\Gamma}(w_{i,k})\leq \liminf_{k\to\infty}\frac{(1+\epsilon_k)w_0(f_k)}{rN_1}=\liminf_{k\to\infty}\frac{w_0(f_k)}{rN_1}.
\]
However, since $(Z,\Gamma_k^+)$ is lc, we always have $A_{Z,\Gamma}(w_0)\geq \frac{w_0(f_k)}{rN_1}$ for $k\gg 1$. Then $w_0(f_k)$ being an integer implies that
\[
A_{Z,\Gamma}(w_0)=\frac{w_0(f_k)}{rN_1} \textrm{ for }k\gg 1.
\]
Therefore, $w_0$ is also an lc place of $(Z,\Gamma_k^+)$ for $k\gg 1$. Denote by $\Gamma_k':=\Gamma+\frac{1}{rN_1}\mathrm{div}(\bin(f_k))$ where $\bin(f_k)$ is the initial degeneration of $f_k$.
% \footnote{\HB{Not sure if this really matters: Do we want $f_k$ to be an element of $R$ not the localization $\cO_{Z,o}$? (It might be more clear what is meant by initial term, if it is an element of $\oplus_j R_j$.)}}
Then by \cite[Theorem 3.1]{dFEM10} we know that $(Z,\Gamma_k')$ is also lc.
% \footnote{\HB{Is this the argument used in Xu19 (originally appearing in Kol08) where one blows up the cone point and applies inversion of adjunction on the exceptional divisor? 
% Maybe, we could add further explanation or say (see the proof of [Kol08, Thm. 32] or the argument written in Xu19)
% }}
Furthermore, by lower semicontinuity of the log discrepancy function, we know that both $w_0$ and $w_{i,k}$ are still lc places of $(Z,\Gamma_k')$ for $k\gg 1$. Hence by taking a $\bG_m$-equivariant resolution, we see that 
\[
A_{Z,\Gamma_k'}(w_{i,k})=\left(1-\frac{mA_{X,\Delta}(v_i)}{k}\right)A_{Z,\Gamma_k'}(w_0)+\frac{c_i}{k}A_{Z,\Gamma_k'}(\ord_{E_{i,\infty}})
\]
which implies that $E_{i,\infty}$ is an lc place of $(Z,\Gamma_k')$ as well. Since $\Gamma_k'$ is $\bG_m$-invariant, it is the cone of some $\bQ$-divisor $\Delta_k'$ on $X$. Hence, we know that $E_i$ is an lc place of $\Delta_k'$ which is a $\bQ$-complement of $(X,\Delta)$. Then by an easy generalization of Theorem \ref{t-globcomp} to the case with multiple divisors over $X$,
%\footnote{\HB{Since \ref{t-globcomp} only covers the case when $d=1$, we could write ``an easy generalization of Theorem \ref{t-globcomp} to the case with multiple divisors over $X$"}} Theorem \ref{t-globcomp},
we may replace $\Delta_k'$ by an $N$-complement $\Delta_k^+$ whose lc places still contain $E_i$ for any $k\gg 1$ and any $1\leq i\leq d$. This finishes proving the ``only if'' part. 
The ``if'' part follows from Proposition \ref{prop:comp-weaklysp}.
\end{proof}

\begin{prop}\label{prop:comp-weaklysp}
Let $(X,\Delta)$ be a log Fano pair. Let $\{v_1,\cdots,v_d\}$ be a set of $\Z$-valued divisorial valuations in $\Val_X$. If $\{v_i\}_{i=1}^d$ is contained in the set of lc places of some $\bQ$-complement $\Delta^+$, then it is a weakly special collection.
\end{prop}

\begin{proof}
%\CX{I think this simpler argument works, but some more words should be added.}
Let $v_i=c_i\cdot \ord_{E_i}$. Then similar as before, we have the the affine cone $o\in (Z,\Gamma)$ over $(X,\Delta)$. For any $k\gg 1$ we have divisorial valuation $w_{i,k}$ and prime divisor $E_{i,k}$ over $Z$ such that $w_{i,k}=\frac{\gcd(k,c_i)}{k}\cdot \ord_{E_{i,k}}$.
Denote by $\Gamma^+$ the Zariski closure of the pull-back of $\Delta^+$ under the projection $Z\setminus\{o\}\to X$.  Then it is clear that $w_{i,k}$ is an lc place of $(Z,\Gamma^+)$ for any $1\leq i\leq d$ and any $k\gg 1$. Hence by \cite{BCHM10}, there exists a $\bG_m$-equivariant projective birational morphism $\tilde{\mu}_k:\widetilde{W}_k\to Z$ from a normal  $\bQ$-factorial variety $\widetilde{W}_k$ such that the following properties hold.
\begin{itemize}
    \item %$\tilde{\mu}_k$ is an isomorphism over $Z\setminus\{o\}$ and 
    The exceptional divisors of $\tilde{\mu}_k$ is $\cup_{i=1}^d \widetilde{E}_{i,k}$ where $\widetilde{E}_{i,k}$ is the birational transform of $E_{i,k}$;
    \item $\widetilde{W}_k$ is of Fano type over $Z$; 
    %\item Each $\widetilde{E}_{i,k}$ is $\bQ$-Cartier on $\widetilde{W}_k$;
    \item $(\widetilde{W}_k, (\tilde{\mu}_k)_*^{-1}\Gamma^+ +\sum_{i=1}^d \widetilde{E}_{i,k})$ is a log canonical crepant model of $(Z,\Gamma^+)$.
\end{itemize}
%It is clear that 
%\begin{equation}\label{eq:weaklysp-equiv}
%-(K_{\widetilde{W}_k}+(\tilde{\mu}_k)_*^{-1}\Gamma+\sum_{i=1}^d \widetilde{E}_{i,k}))\sim_{\tilde{\mu}_k,\bQ} (\tilde{\mu}_k)_*^{-1}(\Gamma^+ -\Gamma)\sim_{\tilde{\mu}_k,\bQ} -\sum_{i=1}^d \ord_{E_{i,k}}(\Gamma^+-\Gamma) \widetilde{E}_{i,k}.
%\end{equation}
%Denote the right-hand side of \eqref{eq:weaklysp-equiv} by $-G$. Then $G$ is an effective $\bQ$-Cartier $\bQ$-divisor fully supported on $\cup_{i=1}^{d}E_{i,k}$ since $(Z,\Gamma)$ is klt. 

%For $0<\epsilon\ll 1$, consider the divisor 
%\[
%\widetilde{P}_{\epsilon}:= (\tilde{\mu}_k)_*^{-1}(\epsilon\Gamma+ (1-\epsilon)\Gamma^+)+\sum_{i=1}^d \widetilde{E}_{i,k} - 2\epsilon G. 
%\]
%Since $(Z,\Gamma)$ is klt, all lc centers of $(\widetilde{W}_k, (\tilde{\mu}_k)_*^{-1} \Gamma^+ +\sum_{i=1}^d \widetilde{E}_{i,k})$ are contained in $\Supp(\Gamma^+-\Gamma)\cup \Exc(\tilde{\mu}_k)$. Thus $(\widetilde{W}_k, \widetilde{P}_\epsilon)$ is a klt pair. 
By \cite{BCHM10}, we could run the $\bG_m$-equivariant $-(K_{\widetilde{W}_k}+(\tilde{\mu}_k)_*^{-1}\Gamma +\sum_{i=1}^d \widetilde{E}_{i,k})$-MMP over $Z$ and
 this MMP yields a birational contraction ${\rho}_k:\widetilde{W}_k\dashrightarrow W'_k$ where $W'_k$ is the log canonical model. For simplicity let us assume that $\rho_k: \widetilde{W}_k\dasharrow W_k'$ precisely contracts $\widetilde{E}_{i',k}$ for $d'<i'\leq d$. Denote by $E_{i,k}':=(\rho_k)_*\widetilde{E}_{i,k}$ for $1\leq i\leq d'$.

Next we will show that $\mu_k':W_k'\to Z$ satisfies (1'), (2'), and (3') in the proof of Proposition \ref{p-extracting}. Since $\tilde{\mu}_k$ is isomorphic in codimension $1$ over $Z\setminus \{o\}$ as $c_{Z}(\widetilde{E}_{i,k})=o$, so is $\mu_k'$. Since $W_k'$ is the log canonical model, we have that $-(K_{W_k'}+(\mu_k')_*^{-1}\Gamma +\sum_{i=1}^{d'} E_{i,k}')$ is ample over $Z$, which implies that $\mu_k'$ is an isomorphism over $Z\setminus\{o\}$ as $-(K_{W_k'}+(\mu_k')_*^{-1}\Gamma +\sum_{i=1}^{d'} E_{i,k}')|_{W_k'\setminus \mu_k'^{-1}(o)}=\mu_k'^*(-(K_{Z}+\Gamma)|_{Z\setminus\{o\}})$. And $(W_k', (\mu_k')_*^{-1}\Gamma +\sum_{i=1}^{d'} E_{i,k}')$ is log canonical since there is a $\bQ$-complement. Thus (1'), (2'), and (3') in the proof of Proposition \ref{p-extracting} hold for $\mu_k'$ from the above arguments.

Next we construct the weakly special test configuration $(\cX',\Delta_{\tc}';\cL')$ by essentially reversing the argument in the proof of Proposition \ref{p-extracting}.
By the proof of Proposition \ref{p-extracting}, we know that 
\begin{equation}\label{eq:E'-normal-bundle}
-(K_{W_k'}+(\mu_k')_*^{-1}\Gamma + \sum_{i=1}^{d'} E_{i,k}')\sim_{\mu_k',\bQ} -\sum^{d'}_{i=1} \frac{k}{m\gcd(k,c_i)}E_{i,k}'
\end{equation}
is ample over $Z$. Hence by taking valuation ideals of the Rees valuations of $\mu_k'$, we know that $W_k'\cong\Proj_{Z}\oplus_{p=0}^\infty I_p$ where $I_p:=\cap_{i=1}^{d'}\fa_{p/k}(w_{i,k})$ is an ideal sequence on $Z$ cosupported at $o$. Since $W_k'$ is the log canonical model of  $-(K_{\widetilde{W}_k}+(\tilde{\mu}_k)_*^{-1}\Gamma +\sum_{i=1}^d \widetilde{E}_{i,k})\sim_{\tilde{\mu}_k, \bQ} - \sum_{i=1}^d \frac{k}{m\gcd(k,c_i)}\widetilde{E}_{i,k}$,
we also have that $I_p=\cap_{i=1}^d \fa_{p/k}(w_{i,k})$. Hence the proof of Proposition \ref{p-extracting} implies that $E_{i',k}$ is an lc place of $(W_k', (\mu_k')_*^{-1}\Gamma + \sum_{i=1}^{d'} E_{i,k}')$ for any $d'<i'\leq d$. Consider the $\Z$-filtration $\cF^\bullet R$ of $R$ defined as $\cF^p R_j:= I_{p+jk}\cap R_j$. Then by the proof of Proposition \ref{p-extracting} we have 
\begin{align*}
\cF^p R_j & = \{s\in R_j\mid v_i(s)\geq p+mA_{X,\Delta}(v_i)j\quad\textrm{ for any } 1\leq i\leq d\}\\ & = \{s\in R_j\mid v_i(s)\geq p+mA_{X,\Delta}(v_i)j\quad\textrm{ for any } 1\leq i\leq d'\}.\label{eq:weaklysp-filt}\numberthis
\end{align*}
Similar to the proof of Proposition \ref{p-extracting}, denote by $\cZ:=\Spec \oplus_{p\in\Z} t^{-p} I_p$ as the $\bG_m$-equivariant degeneration of $Z$ over $\bA^1$. 
Let  $\Gamma_{\cZ}$ be the effective $\bQ$-divisor on $\cZ$ as the Zariski closure of $\Gamma \times (\bA^1\setminus\{0\})$. Then $\cZ$ is normal by integral closedness of $I_p$.

%\textcolor{blue}{YL: Here I imitate the argument in Prop A.3.}

Let $E':=\sum_{i=1}^{d'} \frac{E_{i,k}'}{\gcd(k,c_i)}$ and $E'_{\red}:=\sum_{i=1}^{d'} E_{i,k}'$. Let $l$ be a sufficiently divisible positive integer such that $lE'$ is Cartier on $W_k'$. 
Let $(\cZ',\Gamma_{\cZ'}):=(\cZ,\Gamma_{\cZ})/\bmu_l$. From the proof of Proposition \ref{p-extracting}, we know that $Z_0/\bmu_l$ is isomorphic to the affine cone $C_a(E_{\red}', \cO_{E'_{\red}}(-lE'|_{E'_{\red}}))$. 
Since $(W_k', (\mu_k')_*^{-1}\Gamma + E_{\red}')$ is log canonical, by adjunction we know that $(E_{\red}',\Gamma_{E_{\red}'})$ is slc where $\Gamma_{E_{\red}'}$ is the corresponding different divisor. Moreover, by \eqref{eq:E'-normal-bundle} we have
\[
-(K_{E_{\red}'}+\Gamma_{E_{\red}'})\sim_{\bQ}-(K_{W_k'}+(\mu_k')_*^{-1}\Gamma + E'_{\red})|_{E_{\red}'}\sim_{\bQ} - \frac{k}{m}E'|_{E_{\red}'}
\]
which is ample. Thus $(E_{\red}',\Gamma_{E_{\red}'})$ is a slc log Fano pair which implies that the affine cone $(Z_0,\Gamma_0)/\bmu_l$ is also slc. In particular, we know $Z_0/\bmu_l=(Z_0')_{\red}$ as it is reduced. Since $Z_0/\bmu_l$ is reduced, we know that $Z_0$ is generically reduced, which implies that $Z_0$ is reduced as it is $S_1$ by \cite[Proposition 2.6(ii)]{BHJ17}. Since the quotient map $\sigma: \cZ\to \cZ'$ is \'etale away from central fibers, we have that $K_{\cZ}+\Gamma_{\cZ}+ Z_0 = \sigma^*(K_{\cZ'} +\Gamma_{\cZ'} + (Z_0')_{\red})$. Since $((Z_0')_{\red}, \Gamma_{\cZ'}|_{(Z_0')_{\red}})\cong (Z_0,\Gamma_0)/\bmu_l$ is slc, inversion of adjunction implies that $(\cZ', \Gamma_{\cZ'}+(Z_0')_{\red})$ is log canonical, which implies that $(\cZ, \Gamma_{\cZ}+Z_0)$ is log canonical  by \cite[Proposition 5.20]{KM98}. Thus $(\cZ,\Gamma_{\cZ},\xi;\eta)$ is a weakly special test configuration of $(Z, \Gamma, \xi)$ where $\xi$ (resp. $\eta$) is the vector field on $Z$ (resp. $\cZ$) induced by the grading of $j$ (resp. of $p$). By adjunction we know that $(Z_0, \Gamma_0)$ is slc.

% \textcolor{red}{Since $(W_k', (\mu_k')_*^{-1}\Gamma + \sum_{i=1}^{d'} E_{i,k}')$ is log canonical, by \cite[Lemma 2.21 (2)]{LWX18} we know that $(\cZ,\Gamma_{\cZ},\xi;\eta)$ is a weakly special test configuration of $(Z, \Gamma, \xi)$ where $\xi$ (resp. $\eta$) is the vector field on $Z$ (resp. $\cZ$) induced by the grading of $j$ (resp. of $p$).}
%Hence by the proof of 
% \footnote{\HB{ Should we add ``by the proof of Proposition...'', rather than the statement. (there are a few other places where this could be added)}}
%Proposition \ref{p-extracting} we obtain a weakly special 

Next, we consider the test configuration $(\cX',\Delta_{\tc}';\cL')$ of $(X,\Delta)$ by setting $\cX':=\Proj\bigoplus_{j=0}^\infty \bigoplus_{p=-\infty}^{\infty} t^{-p}\cF^p R_j$ and $\cL'=-m(K_{\cX'}+\Delta_{\tc}')$. Then by \eqref{eq:cone-tc-1} and \eqref{eq:cone-tc-2}  we know that $(Z_0, \Gamma_0)$ is isomorphic to the affine cone over $(\cX_0', \Delta_{\tc,0}'; \cL_0')$. Since $(Z_0, \Gamma_0)$ is slc, we know that $(\cX_0', \Delta_{\tc,0}'; \cL_0')$ is also slc, and hence $(\cX',\Delta_{\tc}';\cL')$ is weakly special by inversion of adjunction. 
Moreover, $v_i=v_{\cX_0'^{(i)}}$ for any $1\leq i\leq d'$ where $(\cX_0'^{(i)})_{1\leq i\leq d'}$ are all the irreducible components of $\cX_0'$.

Finally we construct the desired semiample test configuration $(\cX,\Delta_{\tc};\cL)$ by extracting certain divisors over $\cX'$. Let $F_i$ be the prime divisor over $X\times\bA^1$ as the quasi-monomial combination of $X\times\{0\}$ and $E_i\times\bA^1$ with weights $(1,c_i)$. Then it is clear that $\ord_{F_i}|_{K(X)}=v_i$. We claim that $F_{i'}$ is an lc place of $(\cX',\cX_0'+\Delta_{\tc}')$ for any $d'<i'\leq d$. 
Let $\cY$ be the total space of a test configuration of $(X;L)$ dominating $\cX'$ and $X\times\bA^1$ such that $F_{i'}$ is a divisor on $\cY$ for any $d'< i'\leq d$. Denote by $\pi_1:\cY\to \cX'$ and $\pi_2:\cY\to X\times\bA^1$ the projection morphisms. Set $\cD:=\pi_1^*\cL'-\pi_2^*\cL_{\bA^1}$ where $\cL_{\bA^1}:=-m(K_{X\times\bA^1}+\Delta\times\bA^1)$. By \cite[Lemmas 2.13 and 5.17]{BHJ17} and \eqref{eq:weaklysp-filt}, we have that $\ord_{F_{i'}}(\cD)=-mA_{X,\Delta}(v_{i'})$. On the other hand, from the definition of $\cD$ we see that 
\[
\ord_{F_{i'}}(\cD)=m\big(A_{\cX',\cX_0'+\Delta_{\tc}'}(F_{i'})-A_{X\times\bA^1,X\times\{0\}+\Delta\times\bA^1}(F_{i'})\big).
\]
Since $A_{X\times\bA^1,X\times\{0\}+\Delta\times\bA^1}(F_{i'})=A_{X,\Delta}(v_{i'})$, we know that $A_{\cX',\cX_0'+\Delta_{\tc}'}(F_{i'})=0$. Thus the claim is proved. By \cite{BCHM10}, we can extract the divisors $\{F_{i'}\}_{d'<i'\leq d}$ over $\cX'$ to obtain the desired weakly special semiample test configuration $(\cX,\Delta_{\tc};\cL)$. This finishes the proof.
\end{proof}

\begin{rem}
%\begin{enumerate}
   % \item 
   Applying to the case with a prime divisor, we see that a prime divisor $E$ over $(X,\Delta)$ is weakly special if and only if $\{\ord_E\}$ is a weakly special collection, which is the same as $E$ being an lc place of some $N$-complement. 
   % \item It is an interesting question to characterize which weakly special collection of divisorial valuations corresponds to a weakly special test configuration. Although Theorem \ref{thm:weaklyspecial} gives a necessary condition, an equivalent condition in terms of the geometry of these lc places is not clear to the authors.
%\end{enumerate}
\end{rem}

\subsection{Valuations computing the stability threshold}\label{ss-compute}

In this section, we show any valuation computing $\delta\leq 1$ is the lc place of a bounded complement.
The result may be viewed as a stronger version of Proposition \ref{p-minapprox}.
%\footnote{\HB{I edited this to stronger version of Proposition  \ref{p-minapprox} (it was previously Theorem \ref{t-quasimonomialless1}). }}

\begin{thm}\label{t-minimizer=lcplaceN}
Let $n$ be a natural number and $I\subseteq \Q$
a finite set. There exists a positive integer $N:=N(n,I)$ depending only on $n$ and $I$ satisfying the following:

Assume $(X,\Delta)$ is an $n$-dimensional log Fano pair
such that coefficients of $\Delta$ belong to $I$ and $\d(X,\Delta)\leq 1$.
If $v\in \Val_X^*$ computes $\d(X,\Delta)$, 
then $v$ is quasi-monomial and is an lc place an $N$-complement.
\end{thm}

\begin{proof}%[Proof of Theorem \ref{t-minimizer=lcplaceN}]
Let $(X,\Delta)$ be an $n$-dimensional log Fano pair
such that coefficients of $\Delta$ belong to $I$.
Assume $\d(X,\Delta)\leq 1$ and $v\in \Val_{X}^*$ computes the stability threshold. 
By \cite[Prop. 4.8]{BJ17},  $v$ is the unique valuation (up to scaling) computing  $\lct(X,\Delta; \fa_\bullet(v))$ \cite[Prop. 4.8]{BJ17}. Hence,  $v$ is quasi-monomial by \cite{Xu19}. 

To prove the second part of Theorem \ref{t-minimizer=lcplaceN}, we will again 
use the cone construction.
Fix a positive integer $r$ so that $L: =-r(K_X+\Delta)$ is a Cartier divisor
and set $R: = R(X,L)$. Let $Z=\Spec(R)$ denote the cone over $X$ with respect to the polarization $L$, 
 $o \in Z$  the vertex of the cone, and $\Gamma$ the $\Q$-divisor on $Z$ defined by pulling back $\Delta$.

For each $t\in \R_{\geq 0}$, we consider the valuation $v_t \in \Val_{Z}$
defined by 
\[
v_t( f ) = \min \{  tv(f_m)+ m \, \vert\, f_m \neq 0 \} ,
\]
where $f= \sum f_m$ and each $f_m \in R_m$.
%Observe that $v_t$ is  equivariant with respect to the $\mathbb{G}_m$-action on $Z$.
The valuation $v_t$ is quasi-monomial, since $v$ is quasi-monomial,
and satisfies $A_{Z,\Gamma}(v_t) =\tfrac{1}{r} + t A_{X,\Delta}(v)$ (see the proof of \cite[Lemma 6.14]{Li17}).

\begin{lem}\label{l-lctAvt}
For any $t\in \R_{>0}$, $\lct(Z,\Gamma; \fa_\bullet(v_t)) =A_{Z,\Gamma}(v_t)$.
\end{lem}

\begin{proof}
Since the inequality $\lct(Z,\Gamma; \fa_\bullet(v_t)) \leq A_{Z,\Gamma}(v_t)$ always holds, it suffices to show the reverse inequality. 
Pick any $\varepsilon>0$. We will proceed to show 
$\lct(Z,\Gamma; \fa_\bullet(v_t))  \geq A_{Z,\Gamma}(v_t) -\varepsilon$. 
\medskip

\noindent \emph{Claim}:
For any $\varepsilon' >0$, there exists a $\Q$-complement $\Delta^+$ of $(X,\Delta)$ such that
$A_{X,\Delta^+}(v)<\varepsilon' $. 

To prove the claim, for each $m$ divisible by $r$ 
choose an $m$-basis type divisor $B_m$ such that 
$S_m(v) = v(B_m)$. 
If we set $c_m = \min \{1, \delta_m(X,\Delta)\} $, 
then $(X,\Delta+c_m B_m)$ is lc by the definition of 
$\delta_m$
and $$A_{X,\Delta+c_m B_m}(v)=
A_{X,\Delta}(v)- c_m v(B_m)=
A_{X,\Delta}(v)- c_m S_m(v)
.$$
Since $S_m(v) \to S(v)$ and $c_m \to \delta(X,\Delta)$ as $m\to \infty$, we see
$$
\lim_{m \to \infty }
A_{X,\Delta+c_m B_m}(v)
=
A_{X,\Delta}(v) - \delta(X,\Delta) S(v)
,$$
which is zero since $v$ computes $\delta(X,\Delta)$. 

Therefore, we may find $m$ so that $A_{X,\Delta+c_mB_m}(v) < \varepsilon'$. 
Since $-K_X-\Delta$ is ample, 
we may  choose $H \in  |-K_X-\Delta|_{\Q}$
so that $(X,\Delta+c_mB_m+(1-c_m)H))$ remains lc \cite[Lem. 5.17.2]{KM98}. Hence, 
$\Delta^+: = \Delta+ c_m B_m +(1-c_m) H$ 
is a $\Q$-complement of $(X,\Delta)$ and
satisfies $A_{X,\Delta^+}(v)
\leq A_{X,\Delta+c_m B_m}(v)<\varepsilon'$.
\medskip 

By the above claim, we may choose a $\Q$-complement $\Delta^+$ of $(X,\Delta)$ such that $A_{X,\Delta^+}(v) < \varepsilon/t$. 
Since $\Delta^+- \Delta \sim_{\Q} -K_X-\Delta$, there exists a positive integer $m$ and $f\in H^0(X,\cO_{X}(mL))$ such that $\Delta^+ = \Delta+ \tfrac{1}{mr} \{f=0 \}$. 

Since
$(X,\Delta^+)$ is lc and $K_{X}+\Delta^+\sim_{\Q}0$,
the pair
$(Z, \Gamma + \tfrac{1}{mr}\{ f=0 \})$ is  lc
\cite[Lem. 3.1]{Kol13}. 
Using that $A_{Z,\Gamma}(v_t) = \tfrac{1}{r} + t A_{X,\Delta}(v)$ and  $v_t(f) = m+tv(f)$, 
we see  
\[
 A_{Z,\Gamma+ \frac{1}{mr}\{ f=0 \}}(v_t)
= 
t \left( 
A_{X,\Delta}(v) - \tfrac{1}{mr} v(f)
\right)
= 
t A_{X,\Delta^+}(v)<\varepsilon .
\]
Hence, if we set
$s:=v_t(f)$, then 
\begin{equation}\label{eq-sestimate}
s = mr \big( A_{Z,\Gamma}(v_t) - A_{Z,\Gamma+ \frac{1}{mr} \{f=0\} } (v_t) \big) 
\geq  mr  ( A_{Z,\Gamma}(v_t) - \varepsilon )
.\end{equation}

To estimate the asymptotic lct, 
observe that 
$f^{\lceil p/s \rceil} \in \fa_p(v_t)$
for each positive integer $p$. 
Hence,
\[
\lct( Z,\Gamma; \fa_p(v_t)) 
\geq 
\lct( Z,\Gamma; (f^{\lceil p/s\rceil})  ) 
\geq
\frac{1}{\lceil p/s \rceil mr}
\]
where the last inequality uses that $(Z,\Gamma + \tfrac{1}{mr} \{f=0\} )$ is lc. 
Therefore, 
\[
\lct( Z,\Gamma; \fa_\bullet (v_t)) 
= 
\lim_{p \to \infty} 
p \cdot \lct(Z,\Gamma; \fa_p(v_t))
\geq 
\lim_{p \to \infty} \frac{p}{\lceil p/s \rceil mr}
= \frac{s}{mr}.
\]
After referring back to \eqref{eq-sestimate}, we see 
$\lct( Z,\Gamma; \fa_\bullet (v_t)) \geq A_{Z,\Gamma}(v_t)-\varepsilon$.
\end{proof}

%To proceed to show that $v$ is the lc place of an $N$-complement, we show that each valuation $v_{1/k}$ is the lc place of a bounded local complement.  
%\medskip 

\begin{lem} There exists a positive integer $M$ such that the following holds:
for each positive integer $k$,
there exists  $f^{(k)}\in R$ such that the pair 
$$(Z,\Gamma +\tfrac{1}{rM} \{f^{(k)}=0\})$$
is lc in a neighborhood of $o\in Z$
and $v_{1/k}$ is an lc place of the pair. \end{lem}

\begin{proof}
Fix a positive integer $k$. 
Choose a log resolution 
$W\to Z$ of $(Z,\Gamma)$ 
and local coordinates 
$y_1,\ldots, y_q$
at a point $\eta \in W$
such that $v_{1/k}$ may be written as $v_\alpha$
for some $\alpha = (\alpha_1,\cdots,\alpha_q) \in \R_{> 0}^q$ (see \eqref{eq:qmval} for the definition).
After replacing $W$ with a higher model and choosing new local coordinates, we may assume 
$\alpha_1,\ldots,\alpha_q$ are linearly independent over $\Q$.

Note that $v_{1/k}$ computes $\lct(Z,\Gamma;\fa_\bullet(v_{1/k}))$,
since 
$A_{Z,\Gamma}(v_{1/k}) = \lct ( Z,\Gamma; \fa_\bullet(v_{1/k}))$ by
Lemma \ref{l-lctAvt}
and the equality $v_{1/k}(\fa_\bullet(v_{1/k})) = 1$.
Hence, \cite[Rem. 2.52]{LX17} implies $v_{1/k}$ admits a weak lc model
in the sense of \cite[Def. 2.49]{LX17}. 
In particular, the argument in \textit{loc. cit.} implies there exists a proper birational morphism 
$\rho:W^{\rm wlc} \to Z $, prime divisors $S_1, \ldots, S_q$ on $W^{\rm wlc}$ with $\ord_{S_j} = v_{\beta^{(j)}}$ for some $\beta^{(j)}\in \Z_{\geq 0}^q$ such that 
\begin{enumerate}[label=(\roman*)]
    \item $(W^{\rm wlc} ,\rho_*^{-1} \Gamma +\sum S_i)$ is lc,
    \item $-K_{W^{\rm wlc} }
    - \rho_*^{-1} \Gamma - \sum S_i$ is $\rho$-nef, and
    \item  $\alpha$ lies in the convex cone generated by  $\beta^{(1)},\ldots, \beta^{(q)}$.
\end{enumerate} %\footnote{\HB{maybe this part of the argument could be simplified. I was trying to avoid discussing dual complexes associated to quotients of dlt pairs, which added some length }}

Applying \cite[Thm. 1.8]{Bir16a}, we may find a
positive integer $M$, dependent only on $\dim(Z)$ and the coefficients on $\Gamma$, and  an effective $\Q$-divisor $\Gamma_{W^{\rm wlc}}^+ \geq \rho^{-1}_* (\Gamma)+ \sum S_j$ 
such that, in a neighborhood of $\rho^{-1}(o)$, 
$(W^{\rm wlc}, \Gamma_{W^{\rm wlc}}^+)$ is lc  
and $M(K_{W^{\rm wlc}} +\Gamma_{W^{\rm wlc}}^+)\sim_Z0$.
Set $\Gamma^+ := \rho_* (\Gamma_{W^{\rm wlc}}^+)$.
Observe that
$$\Gamma^+ \geq \Gamma
\quad \text{ and } \quad
K_{W^{\rm wlc}}+ \Gamma_{W^{\rm wlc}}^+ = \rho^*( K_{Z}+\Gamma^+).$$
Hence,  in a neighborhood of $o \in Z$, $(Z,\Gamma^+)$ is lc and $M(K_{Z}+\Gamma^+)\sim 0$.  Additionally, each 
$S_j$ is an lc place of $\Gamma^+$. 

To see $v_{1/k}$ is an lc place of $(Z,\Gamma^+)$, note that
$$A_{Z,\Gamma^+}(w) = A_{Z,\Gamma}(w) - w(\Gamma^+-\Gamma)\quad \text{ for any } w\in \Val_{Z}.$$ 
Since $Y\to Z$ is a log resolution of $(Z,\Gamma)$, $A_{Z,\Gamma}$ is linear on our simplicial cone. 
Additionally, $w\mapsto -w(\Gamma^+-\Gamma)$ is convex on the cone by \cite[Lem. 1.10]{BFJ08}. 
Using that $\alpha$ lies in the convex cone generated by $\beta^{(1)},\ldots, \beta^{(q)}$
and $A_{Z,\Gamma^+}(v_{\beta ^{(j)}})=0$ 
for each $j=1,\ldots, q$,
we see
$A_{Z,\Gamma^+}(v_{1/k}) \leq 0$.
Since $(Z,\Gamma^+)$ is lc, we conclude 
$A_{Z,\Gamma^+}(v_{1/k}) =0$. 

Finally, note that $$
rM(\Gamma^+-\Gamma) 
= rM(K_Z+\Gamma^+) 
- rM(K_Z+\Gamma) \sim 0 
$$
at $o\in Z$.
Hence, we may find $f^{(k)}\in R$ such that $\Gamma^+-\Gamma$ agrees with $\frac{1}{Mr}\{f^{(k)}=0\}$
in a neighborhood of $o\in Z$, which completes the proof of the lemma. 
\end{proof}

For each positive integer $k$, consider the lc pair $(Z,\Gamma+ \tfrac{1}{Mr}\{f^{(k)} =0 \})$ constructed above.
Repeating the proof of Theorem \ref{thm:weaklyspecial}, we see that if $k\gg0$, then $\Delta^+: = \Delta+ \tfrac{1}{Mr} \{ \bin(f^{(k)}) =0 \})$ is a $\Q$-complement of $(X,\Delta)$ with $v$ an lc place.

To show $v$ is the lc place of an $N:=N(n,I)$ complement, let $\pi:Y\to X$ be a log resolution of $(X,\Delta^+)$ and write $\Delta^+_Y$ for the $\Q$-divisor satisfying $$K_{Y}+ \Delta_Y^+= \pi^*(K_X+\Delta^+).$$
By Lemma \ref{l-lcplacesQM}, the lc places of $(X,\Delta^+)$ coincide with the simplicial cone complex $\QM(Y,(\Delta_Y^+)^{=1})$. 

Choose a sequence of divisorial valuations
$(v_j)_j$ in $\QM(Y,(\Delta_Y^+)^{=1})$
converging to $v$. 
Since  each $v_j$ is divisorial and an lc place of  $(X,\Delta^+)$, 
Theorem \ref{t-globcomp} implies there exists a positive integer $N:=N(n,I)$,  depending only on $n$ and $I$,
such that each $v_j$ is the lc place of an $N$ complement. 
Hence, for each $j$, we may choose an $N$-complement  $\Delta_j^+$ with $A_{X,\Delta^+_j}(v_j)=0$.

Set $D_j := \Delta^+_j - \Delta$
and write $\varphi_{D_j} : \QM(Y,(\Delta_Y^+)^{=1})\to \R$ for the function defined by $v\mapsto v(D_j)$. 
Since each $D_j$ is an element of  $\frac{1}{rN}|-rN(K_X+\Delta)|$,
the set of functions
$\{ \varphi_{D_j} \, \vert\, j \in \N\}$ 
is finite by Lemma \ref{lem:linearseriesQMfinite}. 
Therefore, after replacing $(v_j)_j$ with a subsequence, 
we may find an individual  $N$-complement $\Delta^+_*$ such that $A_{X,\Delta_*^{+}}(v_{j})=0$ for all $j$.
Using that $v= \lim_j v_j$, we  conclude $A_{X,{\Delta^+_*}}(v)=0$.
\end{proof}

\bigskip


\begin{bibdiv}
\begin{biblist}%[\normalsize]



\bib{ABHLX19}{article}{
   AUTHOR={Alper, J.},
   AUTHOR={Blum, H.},
   AUTHOR={Halpern-Leistner, D.},
   AUTHOR={Xu, C.},
   title={Reductivity of the automorphism group of K-polystable Fano
   varieties},
   journal={Invent. Math.},
   volume={222},
   date={2020},
   number={3},
   pages={995--1032},
}
		
  
%   \bib{AHR20}{article}{
%       AUTHOR = {Alper, Jarod},
%       Author = {Hall, Jack},
%       Author = {Rydh, David},
%      TITLE = {A {L}una \'{e}tale slice theorem for algebraic stacks},
%   JOURNAL = {Ann. of Math. (2)},
%   FJOURNAL = {Annals of Mathematics. Second Series},
%     VOLUME = {191},
%       YEAR = {2020},
%     NUMBER = {3},
%      PAGES = {675--738},
%      }
   
   \bib{Amb16}{article}{
    AUTHOR = {Ambro, Florin},
     TITLE = {Variation of log canonical thresholds in linear systems},
   JOURNAL = {Int. Math. Res. Not. IMRN},
  FJOURNAL = {International Mathematics Research Notices. IMRN},
      YEAR = {2016},
    NUMBER = {14},
     PAGES = {4418--4448},
}



 \bib{Bir16a}{article}{
 AUTHOR = {Birkar, Caucher},
     TITLE = {Anti-pluricanonical systems on {F}ano varieties},
   JOURNAL = {Ann. of Math. (2)},
  FJOURNAL = {Annals of Mathematics. Second Series},
    VOLUME = {190},
      YEAR = {2019},
    NUMBER = {2},
     PAGES = {345--463},
}


 \bib{Bir16b}{article}{
 AUTHOR={Birkar, C.}, 
   title={Singularities of linear systems and boundedness of Fano varieties},
   journal={Ann. of Math. (2)},
   volume={193},
   date={2021},
   number={2},
   pages={347--405},
 }




\bib{BCHM10}{article}{
   author={Birkar, C.},
  author={Cascini, P.},
  author={Hacon, C.},
  author={McKernan, J.},
 title={Existence of minimal models for varieties of log general type},
  journal={J. Amer. Math. Soc.},
  volume={23},
   date={2010},
   number={2},
   pages={405--468},
%   issn={0894-0347},
%   review={\MR{2601039 (2011f:14023)}},
%   doi={10.1090/S0894-0347-09-00649-3},
}



%\bib{BHJ16}{article}{
 %   AUTHOR = {Boucksom, S.} 
  %  AUTHOR={Hisamoto, T.} 
  %  AUTHOR={ Jonsson, M.},
  %   TITLE = {Uniform {K}-stability, and asymptotics of energy functionals in K\"ahler geometry},
 %  JOURNAL = {arXiv:1603.01026, to appear in J. Euro. Math. Soci.},
   %   YEAR = {2016},
 %}
 
 \bib{BHLLX20}{article}{
     AUTHOR = {Blum, H.},
     AUTHOR = {Halpern-Leistner, D.},
    AUTHOR = {Liu, Y.},
    AUTHOR = {Xu, C.},
   title={On properness of K-moduli spaces and optimal degenerations of Fano
   varieties},
   journal={Selecta Math. (N.S.)},
   volume={27},
   date={2021},
   number={4},
   pages={Paper No. 73},
}

 

\bib{BJ17}{article}{
   author={Blum, H.},
   AUTHOR={Jonsson, M.}, 
     TITLE = {Thresholds, valuations, and K-stability},
   JOURNAL = {Adv. Math.},
  FJOURNAL = {Advances in Mathematics},
    VOLUME = {365},
      YEAR = {2020},
     PAGES = {107062},
   }

\bib{Blu18}{article}{
    AUTHOR = {Blum, H.},
    TITLE = {Singularities and K-stability},
   JOURNAL = {University of Michigan Ph.D. Thesis},
    YEAR = {2018},
 }

\bib{BL18}{article}{
   author={Blum, H.},
   AUTHOR={Liu, Y.},
  TITLE = {Openness of uniform K-stability in families of $\mathbb{Q}$-Fano varieties},
   JOURNAL = {arXiv:1808.09070, to appear in Ann. Sci. \'{E}c. Norm. Sup\'{e}r.},
      YEAR = {2018},
   }	
 
 \bib{BLZ19}{article}{
   author={Blum, H.},
  AUTHOR={Liu, Y.},
    AUTHOR={Zhou, C.},
   TITLE = {Optimal destabilization of K-unstable Fano varieties via stability thresholds},
   JOURNAL = {arXiv:1907.05399, to appear in Geom. Topol.},
     YEAR = {2019},
  }	
  
  

\bib{BX18}{article}{
   AUTHOR = {Blum, H.},
   AUTHOR={Xu, C.},
     TITLE = {Uniqueness of K-polystable degenerations of
              {F}ano varieties},
   JOURNAL = {Ann. of Math. (2)},
  FJOURNAL = {Annals of Mathematics. Second Series},
    VOLUME = {190},
      YEAR = {2019},
    NUMBER = {2},
     PAGES = {609--656},
}

 \bib{BFJ08}{article}{
    AUTHOR = {Boucksom, S.},
    AUTHOR = {Favre, C.},
    AUTHOR = {Jonsson, M.},
     TITLE = {Valuations and plurisubharmonic singularities},
   JOURNAL = {Publ. Res. Inst. Math. Sci.},
  FJOURNAL = {Kyoto University. Research Institute for Mathematical
              Sciences. Publications},
    VOLUME = {44},
      YEAR = {2008},
    NUMBER = {2},
     PAGES = {449--494},
      ISSN = {0034-5318},
%   MRCLASS = {32U25 (13A18 14B05 32U05)},
%  MRNUMBER = {2426355},
%MRREVIEWER = {Romain Dujardin},
   %    DOI = {10.2977/prims/1210167334},
    %   URL = {https://doi.org/10.2977/prims/1210167334},
}
	



\bib{BdFFU15}{incollection}{
    AUTHOR = {Boucksom, S.},
    AUTHOR={de Fernex, T. },
    AUTHOR={ Favre, C.}, 
    AUTHOR={Urbinati, S.},
     TITLE = {Valuation spaces and multiplier ideals on singular varieties},
 BOOKTITLE = {Recent advances in algebraic geometry},
    SERIES = {London Math. Soc. Lecture Note Ser.},
    VOLUME = {417},
     PAGES = {29--51},
 PUBLISHER = {Cambridge Univ. Press, Cambridge},
      YEAR = {2015},
}

\bib{BHJ17}{article}{
    AUTHOR = {Boucksom, S.},
    AUTHOR={Hisamoto, T.}, 
    AUTHOR={ Jonsson, M.},
     TITLE = {Uniform {K}-stability, {D}uistermaat-{H}eckman measures and
              singularities of pairs},
   JOURNAL = {Ann. Inst. Fourier (Grenoble)},
  FJOURNAL = {Universit\'e de Grenoble. Annales de l'Institut Fourier},
    VOLUME = {67},
      YEAR = {2017},
     PAGES = {743--841},
 }


 \bib{BoJ18}{article}{
    AUTHOR = {Boucksom, S.}, 
    AUTHOR={ Jonsson, M.},
     TITLE = {A non-Archimedean approach to K-stability},
  JOURNAL = {arXiv:1805.11160},
      YEAR = {2018},
 }


% \bib{CDS}{article}{
%     AUTHOR = {Chen, X.},
%     AUTHOR={Donaldson, S.}, 
%     AUTHOR={Sun, S.},
%      TITLE = {K\"ahler-{E}instein metrics on {F}ano manifolds. {I}:
%               {A}pproximation of metrics with cone singularities, II: Limits with cone angle less than $2\pi$, III: Limits as cone angle approaches $2\pi$ and completion of the main proof.},
%   JOURNAL = {J. Amer. Math. Soc.},
%   FJOURNAL = {Journal of the American Mathematical Society},
%     VOLUME = {28},
%       YEAR = {2015},
%     NUMBER = {1},
%      PAGES = {183--197, 199--234, 235--278},
%  }


\bib{dFEM10}{article}{
   author={de Fernex, Tommaso},
   author={Ein, Lawrence},
   author={Musta\c{t}\u{a}, Mircea},
   title={Shokurov's ACC conjecture for log canonical thresholds on smooth
   varieties},
   journal={Duke Math. J.},
   volume={152},
   date={2010},
   number={1},
   pages={93--114},
%   issn={0012-7094},
%   review={\MR{2643057}},
%   doi={10.1215/00127094-2010-008},
}

% \bib{dFKX}{incollection}{
%   author={de Fernex, T.},
%      author={Koll\'ar, J.},
%       author={Xu, C.},
%   title={The dual complex of singularities},
%     journal={Inst. Hautes \'Etudes Sci. Publ. Math.},
%   number={40},
%   date={2012},
%  %  note={arXiv:1212.1675},
%   booktitle={Proceedings of the conference in honor of Yujiro Kawamata's 60th birthday, Advanced Studies in Pure Mathematics.},
%   series={Advanced Studies in Pure Mathematics},
%   pages={5--57},
   
% }


	
  \bib{Don02}{article}{
    AUTHOR = {Donaldson, S. K.},
     TITLE = {Scalar curvature and stability of toric varieties},
   JOURNAL = {J. Differential Geom.},
    VOLUME = {62},
      YEAR = {2002},
    NUMBER = {2},
     PAGES = {289--349},
    }




    \bib{Fuj19}{article}{
  author={Fujita, K.},
   title={A valuative criterion for uniform K-stability of $\mathbb{Q}$-Fano varieties},
   journal={ J. Reine Angew. Math.},
    VOLUME = {751},
     date={2019},
     pages={309-338}
}


\bib{Fuj17}{article}{
AUTHOR={Fujita, K.},
TITLE = {Uniform K-stability and plt blowups of log Fano pairs},
JOURNAL = {Kyoto J. Math.},
  VOLUME = {59},
 YEAR = {2019},
    NUMBER = {2},
     PAGES = {399-418},
}

\bib{FO18}{article}{
    AUTHOR = {Fujita, K.},
    AUTHOR = {Odaka, Y.},
     TITLE = {On the {K}-stability of {F}ano varieties and anticanonical
              divisors},
   JOURNAL = {Tohoku Math. J. (2)},
    VOLUME = {70},
      YEAR = {2018},
    NUMBER = {4},
     PAGES = {511--521},
}


\bib{HMX13}{article}{
    AUTHOR = {Hacon, C.},
    AUTHOR = {McKernan, J.},
    AUTHOR = {Xu, C.},
   title={On the birational automorphisms of varieties of general type},
   journal={Ann. of Math. (2)},
   volume={177},
   date={2013},
   number={3},
   pages={1077--1111},
}


	
\bib{HLS19}{article}{
   AUTHOR={Han, J.},
    AUTHOR={Liu, J.},
    AUTHOR={Shokurov, V. V.},
     TITLE = {ACC for minimal log discrepancies of exceptional singularities},
   JOURNAL = {arXiv:1903.04338},
      YEAR = {2019},
   }

   \bib{Jia19}{article}{
    AUTHOR = {Jiang, C.},
     TITLE = {Boundedness of $\mathbb{Q}$-Fano varieties with degrees and alpha-invariants bounded from below},
   journal={Ann. Sci. \'{E}c. Norm. Sup\'{e}r. (4)},
   volume={53},
   date={2020},
   number={5},
   pages={1235--1248},
   }
   

\bib{JM12}{article}{
    AUTHOR = {Jonsson, M.},
    AUTHOR={Musta\c t\u a, M.},
     TITLE = {Valuations and asymptotic invariants for sequences of ideals},
   JOURNAL = {Ann. Inst. Fourier (Grenoble)},
  FJOURNAL = {Universit\'e de Grenoble. Annales de l'Institut Fourier},
    VOLUME = {62},
      YEAR = {2012},
    NUMBER = {6},
     PAGES = {2145--2209 (2013)},
   }






	

	 \bib{Kol13}{book}{
    AUTHOR = {Koll{\'a}r, J.},
     TITLE = {Singularities of the minimal model program},
    series= {Cambridge Tracts in Math.},
    VOLUME = {200},
      NOTE = {With the collaboration of S{\'a}ndor Kov{\'a}cs},
 PUBLISHER = {Cambridge Univ. Press},
   ADDRESS = {Cambridge},
      YEAR = {2013},
   }
   
	 \bib{Kol19}{book}{
    AUTHOR = {Koll{\'a}r, J.},
     TITLE = {Families of varieties of general type},
    series= {Book to appear},
      note={See \href{https://web.math.princeton.edu/~kollar/book/modbook20170720-hyper.pdf}{\textsf{Online preprint}}},
       YEAR = {2020},
   }   
  
\bib{KM98}{book}{
   author={Koll{\'a}r, J.},
   author={Mori, S.},
   title={Birational geometry of algebraic varieties},
   series={Cambridge Tracts in Math.},
   volume={134},
   note={With the collaboration of C. H. Clemens and A. Corti},
    publisher={Cambridge Univ. Press},
   place={Cambridge},
   date={1998},
   pages={viii+254},
}




  \bib{Li17}{article}{
  author={Li, C.},
   title={ K-semistability is equivariant volume minimization},
 JOURNAL = {Duke Math. J.},
  FJOURNAL = {Duke Mathematical Journal},
    VOLUME = {166},
      YEAR = {2017},
    NUMBER = {16},
     PAGES = {3147--3218},
      ISSN = {0012-7094},
}

   \bib{Li18}{article}{
  author={Li, C.},
   title={Minimizing normalized volumes of valuations},
   journal={Math. Zeit.},
    VOLUME = {289},
      YEAR = {2018},
    NUMBER = {1-2},
     PAGES = {491--513},
  date={2018},
}

 
% \bib{LLX18}{article}{
%     AUTHOR = {Li, C.},
%   AUTHOR={Liu, Y.},
%     AUTHOR={Xu, C.},
%      TITLE = {A guided tour to normalized volumes},
%   JOURNAL = { arXiv:1806.07112},
%       YEAR = {2018},
% }



 \bib{LWX19}{article}{
  author={Li, C.},
  author={Wang, X.},
  author={Xu, C.},
   title={ On the proper moduli spaces of smoothable K\"ahler-Einstein Fano varieties},
   journal={Duke Math. J.},
   VOLUME = {168},
 YEAR = {2019},
    NUMBER = {8},
     PAGES = {1387-1459},
}     

\bib{LWX18}{article}{
author={Li, C.},
author={Wang, X.},
author={Xu, C.},
title={Algebraicity of the metric tangent cones and equivariant K-stability},
%   journal={arXiv:1805.03393, to appear in J. Amer. Math. Soc.},
    journal = {J. Amer. Math. Soc.},
   volume={34},
   date={2021},
   number={4},
   pages={1175--1214},
}

\bib{LX14}{article}{
    AUTHOR = {Li, C.},
    AUTHOR={Xu, C.},
     TITLE = {Special test configuration and {K}-stability of {F}ano       varieties},
   JOURNAL = {Ann. of Math. (2)},
     VOLUME = {180},
      YEAR = {2014},
    NUMBER = {1},
     PAGES = {197--232},
}



\bib{LX17}{article}{
    AUTHOR = {Li, C.},
    AUTHOR={Xu, C.},
     TITLE = { Stability of Valuations: Higher Rational Rank },
       VOLUME = {1},
    NUMBER = {1},
     PAGES = {1--79},
 JOURNAL = {Peking Math. J. },
      YEAR = {2018},
}


\bib{LX16}{article}{
    AUTHOR = {Li, C.},
    AUTHOR={Xu, C.},
     TITLE = { Stability of Valuations and Koll\'ar Components },
        JOURNAL = {J. Eur. Math. Soc. (JEMS)},
  FJOURNAL = {Journal of the European Mathematical Society (JEMS)},
    VOLUME = {22},
      YEAR = {2020},
    NUMBER = {8},
     PAGES = {2573--2627},
}

\bib{LXZ21}{article}{
    AUTHOR = {Liu, Y.},
    AUTHOR = {Xu, C.}
    AUTHOR = {Zhuang, Z.}
    TITLE = {Finite generation for valuations computing stability thresholds and applications to K-stability}
    JOURNAL = {arXiv:2102.09405}
    YEAR = {2021}
}



%  \bib{Liu18}{article}{
%   author={Liu, Y.},
%   TITLE = {The volume of singular {K}\"ahler--{E}instein {F}ano varieties},
%   JOURNAL = {Compos. Math.},
%   FJOURNAL = {Compositio Mathematica},
%     VOLUME = {154},
%       YEAR = {2018},
%     NUMBER = {6},
%      PAGES = {1131--1158},}

% \bib{Lun73}{article} {
%     AUTHOR = {Luna, Domingo},
%      TITLE = {Slices \'{e}tales},
%  BOOKTITLE = {Sur les groupes alg\'{e}briques},
%      PAGES = {81--105. Bull. Soc. Math. France, Paris, M\'{e}moire 33},
%       YEAR = {1973},
% }


\bib{Muk03}{book}{
    AUTHOR = {Mukai, Shigeru},
     TITLE = {An introduction to invariants and moduli},
    SERIES = {Cambridge Studies in Advanced Mathematics},
    VOLUME = {81},
      NOTE = {Translated from the 1998 and 2000 Japanese editions by W. M.
              Oxbury},
 PUBLISHER = {Cambridge University Press, Cambridge},
      YEAR = {2003},
     PAGES = {xx+503},
      ISBN = {0-521-80906-1},
}


\bib{GIT}{book}{
    AUTHOR = {Mumford, D.},
    Author = {Fogarty, J.},
    Author = {Kirwan, F.},
     TITLE = {Geometric invariant theory},
    SERIES = {Ergebnisse der Mathematik und ihrer Grenzgebiete (2) [Results
              in Mathematics and Related Areas (2)]},
    VOLUME = {34},
   EDITION = {Third},
 PUBLISHER = {Springer-Verlag, Berlin},
      YEAR = {1994},
     PAGES = {xiv+292},
      ISBN = {3-540-56963-4},
}


\bib{Poi13}{article}{
    AUTHOR = {Poineau, J.},
     TITLE = {Les espaces de {B}erkovich sont ang\'{e}liques},
   JOURNAL = {Bull. Soc. Math. France},
  FJOURNAL = {Bulletin de la Soci\'{e}t\'{e} Math\'{e}matique de France},
    VOLUME = {141},
      YEAR = {2013},
    NUMBER = {2},
     PAGES = {267--297},
}


	



% \bib{Oda15}{article}{
 %   AUTHOR = {Odaka, Y.},
 %    TITLE = {Compact moduli spaces of {K}\"ahler-{E}instein {F}ano varieties},
 %  JOURNAL = {Publ. Res. Inst. Math. Sci.},
 % FJOURNAL = {Publications of the Research Institute for Mathematical         Sciences},
  %  VOLUME = {51},
   %   YEAR = {2015},
  %  NUMBER = {3},
   %  PAGES = {549--565},}

\bib{Sho92}{article}{
    AUTHOR = {Shokurov, V. V.},
     TITLE = {Three-dimensional log perestroikas},
   JOURNAL = {Izv. Ross. Akad. Nauk Ser. Mat.},
  FJOURNAL = {Rossi\u{\i}skaya Akademiya Nauk. Izvestiya. Seriya
              Matematicheskaya},
    VOLUME = {56},
      YEAR = {1992},
    NUMBER = {1},
     PAGES = {105--203},
      ISSN = {1607-0046},
}
 

\bib{Tia87}{article}{
   author={Tian, Gang},
   title={On K\"{a}hler-Einstein metrics on certain K\"{a}hler manifolds with
   $C_1(M)>0$},
   journal={Invent. Math.},
   volume={89},
   date={1987},
   number={2},
   pages={225--246},
   issn={0020-9910},
}
	

 \bib{Tia97}{article}{
    AUTHOR = {Tian, Gang},
     TITLE = {K\"ahler-{E}instein metrics with positive scalar curvature},
   JOURNAL = {Invent. Math.},
    VOLUME = {130},
      YEAR = {1997},
    NUMBER = {1},
     PAGES = {1--37},
}
	

% \bib{Tia15}{article}{
%     AUTHOR = {Tian, G.},
%      TITLE = {K-stability and {K}\"ahler-{E}instein metrics},
%   JOURNAL = {Comm. Pure Appl. Math.},
%   FJOURNAL = {Communications on Pure and Applied Mathematics},
%     VOLUME = {68},
%       YEAR = {2015},
%     NUMBER = {7},
%      PAGES = {1085--1156},
%       }

\bib{Xu19}{article}{
author={Xu, C.},
title={A minimizing valuation is quasi-monomial},
journal={Ann. of Math. (2)},
   VOLUME = {191},
      YEAR = {2020},
    NUMBER = {3},
     PAGES = {1003--1030},
}

\bib{Xu20}{article}{
author={Xu, C.},
   title={K-stability of Fano varieties: an algebro-geometric approach},
   journal={EMS Surv. Math. Sci.},
   volume={8},
   date={2021},
   number={1-2},
   pages={265--354},
}



\bib{XZ19}{article}{
author={Xu, C.},
author={Zhuang, Z.},
   title={On positivity of the CM line bundle on K-moduli spaces},
   journal={Ann. of Math. (2)},
   volume={192},
   date={2020},
   number={3},
   pages={1005--1068},
}

\bib{ZZ19}{article}{
author={Zhou, C.},
author={Zhuang, Z.},
title={Some criteria for uniform K-stability},
journal={arXiv:1907.05293, to appear in Math. Res. Lett.},
year={2019},
}


\end{biblist}
\end{bibdiv}
\end{document}